\documentclass[a4paper]{amsart}                % American Mathematical Society document class for articles
\usepackage[latin1]{inputenc}                  % Codificaci�n europea del teclado (acentos, s�mbolos teclado ...)
\usepackage[T1]{fontenc}                       % Acentos de palabras y divisi�n sil�bica
\usepackage[spanish,english]{babel}            % Idioma principal el ingles
\usepackage{amsmath,amssymb,amsthm}            % Paquetes de la American Mathematical Society
\usepackage{latexsym}                          % S�mbolos
\usepackage{delarray}                          % Delimiter array: permite integrar delimitadores en las opciones del entorno array
\usepackage{bbm}                               % Fuentes para los n�meros reales, racionales, ...
\usepackage{hyperref}                          % Enlaces de "hipertexto"
\usepackage[pdftex,usenames,dvipsnames]{color} % Muy importante para el color
\usepackage{bbding,trfsigns}                   % S�mbolos
\usepackage{wasysym}                           % Variar el tama�o de los s�mbolos
\usepackage{datetime}                          % Imprimir fecha y hora

\DeclareMathAlphabet{\mathpzc}{OT1}{pzc}{m}{it} % Tipo de letra \mathpzc

\newtheorem{Th}{Theorem}[section]              % Enumera los teoremas de acuerdo con la secci�n (Theorem 1.1, Theorem 1.2 , ...)
\newtheorem{Prop}[Th]{Proposition}

\newcommand{\B}{\mathbb{B}}
\newcommand{\E}{\mathbb{E}}
\newcommand{\G}{\Gamma}
\newcommand{\LL}{\mathcal{L}}
\newcommand{\R}{{\mathbb{R}^n}}
\newcommand{\T}{\mathbb{T}}
\newcommand{\Y}{\mathcal{Y}}
\DeclareMathOperator{\supp}{supp}
\DeclareMathOperator{\spann}{span}

\title[$\gamma$-radonifying operators in the Hermite setting on BMO and Hardy spaces]
      {$\gamma$-radonifying operators and UMD-valued Littlewood-Paley-Stein functions in the Hermite setting on BMO and Hardy spaces}

\author[J.J. Betancor]{J.J. Betancor}
\author[A.J. Castro]{A.J. Castro}
\author[J. Curbelo]{J. Curbelo}
\author[J.C. Fari\~na]{J.C. Fari\~na}
\author[L. Rodr\'{\i}guez-Mesa]{L. Rodr\'{\i}guez-Mesa}
\address{\newline
        Jorge J. Betancor, Alejandro J. Castro, Juan C. Fari\~na and Lourdes Rodr\'{\i}guez-Mesa \newline
        Departamento de An\'alisis Matem\'atico,
        Universidad de la Laguna, \newline
        Campus de Anchieta, Avda. Astrof\'{\i}sico Francisco S\'anchez, s/n, \newline
        38271, La Laguna (Sta. Cruz de Tenerife), Spain}
\email{jbetanco@ull.es, ajcastro@ull.es, jcfarina@ull.es, lrguez@ull.es}

\address{\newline
        Jezabel Curbelo \newline
        Instituto de Ciencias Matem\'aticas (CSIC-UAM-UCM-UC3M), \newline
        Nicol\'as Cabrera, 13-15, \newline
        28049, Madrid, Spain}
\email{jezabel.curbelo@icmat.es}



\thanks{The authors are partially supported by MTM2010/17974.
        The second author is also supported by a FPU grant from the Government of Spain.
        The third author is supported by a grant JAE-Predoc of the CSIC (Spain).}

\begin{document}

%  \footnotetext{Date: \today.}

  \maketitle                % Si no se activa esta opci�n no se pone ni el t�tulo ni los autores en el encabezado de cada p�gina

  \begin{abstract}
    In this paper we study Littlewood-Paley-Stein functions associated with the Poisson semigroup for the Hermite
    operator on functions with values in a UMD Banach space $\B.$ If we denote by $H$ the Hilbert space
    $L^2((0,\infty),dt/t),$ $\gamma(H,\B)$ represents the space of $\gamma$-radonifying operators from $H$ into $\B.$
    We prove that the Hermite square function defines bounded operators from $BMO_\mathcal{L}(\R,\B)$
    (respectively, $H^1_\mathcal{L}(\R, \B)$) into $BMO_\mathcal{L}(\R,\gamma(H,\B))$
    (respectively, $H^1_\mathcal{L}(\R, \gamma(H,\B))$), where $BMO_\mathcal{L}$ and $H^1_\mathcal{L}$
    denote $BMO$ and Hardy spaces in the Hermite setting. Also, we obtain equivalent norms in
    $BMO_\mathcal{L}(\R, \B)$ and $H^1_\mathcal{L}(\R,\B)$ by using Littlewood-Paley-Stein functions.
    As a consequence of our results, we establish new characterizations of the UMD Banach spaces.
  \end{abstract}

    %%%%%%%%%%%%%%%%%%%%%%%%%%%%%%%%%%%%%%%%%%%%%%%%%%%%%%%%%%%%%%%%%%%%%%%%%%%%%%%%%%%%%%%%%%%%%%%%%%%%%%%%%%%%%%%%%%%%
    \section{Introduction} \label{sec:intro}
    %%%%%%%%%%%%%%%%%%%%%%%%%%%%%%%%%%%%%%%%%%%%%%%%%%%%%%%%%%%%%%%%%%%%%%%%%%%%%%%%%%%%%%%%%%%%%%%%%%%%%%%%%%%%%%%%%%%%

    The Littlewood-Paley-Stein $g$-function associated with the classical Poisson semigroup $\{ P_t\}_{t>0}$ is given by
    \begin{equation*}\label{1.0}
        g(\{ P_t\}_{t>0})(f)(x)=\left( \int_0^\infty |t\partial_t P_t f(x)|^2 \frac{dt}{t}\right)^{1/2}, \quad x \in \R.
    \end{equation*}
    It is well-known that this $g$-function defines an equivalent norm in $L^p(\R)$, $1<p<\infty.$ Indeed, for every $1<p<\infty$ there exists $C_p>0$ such that
    \begin{equation}\label{1.1}
        \frac{1}{C_p}\|f\|_{L^p(\R)}\leq \|g(\{ P_t\}_{t>0})(f)\|_{L^p(\R)}\leq C_p \|f\|_{L^p(\R)}, \quad f\in L^p(\R).
    \end{equation}
    Equivalence \eqref{1.1} is useful, for instance, to study $L^p$-boundedness properties of certain type of spectral multipliers.\\

    In \cite{Ste} $g$-functions associated with diffusion semigroups $\{ T_t\}_{t>0}$ on
    the measure space $(\Omega,\mu)$ were considered. In this general case \eqref{1.1} takes the following form, for every $1<p<\infty$,
    $$ \frac{1}{C_p}\|f-E_0(f)\|_{L^p(\Omega,\mu)}
        \leq \|g(\{ T_t\}_{t>0})(f)\|_{L^p(\Omega,\mu)}
        \leq C_p \|f\|_{L^p(\Omega,\mu)},
        \quad f\in L^p(\Omega,\mu),$$
    where $C_p>0.$ Here $E_0$ is the projection onto the fixed point space of $\{ T_t\}_{t>0}.$\\

    Suppose that $\B$ is a Banach space. For every $1<p<\infty,$ we denote by $L^p(\R,\B)$ the $p$-Bochner-Lebesgue space.
    The natural way of extending the definition of $g(\{ P_t\}_{t>0})$ to $L^p(\R,\B),$ $1<p<\infty,$ is the following
    $$g_\B(\{ P_t\}_{t>0})(f)(x)=\left( \int_0^\infty \|t\partial_t P_t f(x)\|_\B^2 \frac{dt}{t}\right)^{1/2}, \quad f\in L^p(\R,\B), \ 1<p<\infty.$$
    Kwapie{\'n} in \cite{Kw} proved that $\B$ is isomorphic to a Hilbert space if only if
    \begin{equation}\label{1.2}
        \|f\|_{L^p(\R,\B)}\sim~\|g_\B(\{P_t \}_{t>0})(f)\|_{L^p(\R)},\quad f\in L^p(\R,\B),
    \end{equation}
    for some (or equivalently, for any) $1<p<\infty.$\\

    Xu \cite{Xu} considered generalized $g$-functions defined by
    $$ g_{\B,q}(\{P_t\}_{t>0})(f)(x)=\left( \int_0^\infty \|t\partial_tP_t(f)(x)\|_\B^q \frac{dt}{t}\right)^{1/q}, \quad f\in L^p(\R,\B),  \ 1<p<\infty,$$
    where $1<q<\infty.$ He characterized those Banach spaces $\B$ for which one of the following inequalities holds
    \begin{itemize}
        \item $ \displaystyle \|g_{\B,q}(\{P_t\}_{t>0})(f)\|_{L^p(\R)}\leq C\|f\|_{L^p(\R,\B)}, \quad f\in L^p(\R,\B),  \ 1<p<\infty, $\\

        \item $ \displaystyle \|f\|_{L^p(\R,\B)}\leq C \|g_{\B,q}(\{P_t \}_{t>0})(f)\|_{L^p(\R)}, \quad f\in L^p(\R,\B),  \ 1<p<\infty.$
    \end{itemize}
    The validity of these inequalities is characterized by the martingale type or cotype of the Banach space $\B.$\\

    Xu's results were extended to diffusion semigroups by Mart\'{\i}nez, Torrea and Xu \cite{MTX}.\\

    In order to get new equivalent norms in $L^p(\R,\B)$ for a wider class of
    Banach spaces, Hyt\"onen \cite{Hy1} and Kaiser and Weis (\cite{K} and \cite{KW1}) have introduced new definitions
    of $g$-functions for Banach valued functions.\\

    In this paper we are motivated by the ideas developed by Kaiser and Weis (\cite{K} and \cite{KW1}).
    They defined $g$-functions for Banach valued functions by using $\gamma$-radonifying operators.\\

    The main definitions and properties about $\gamma$-radonifying operators can be found in \cite{Nee}.
    We now recall those aspects of the theory of $\gamma$-radonifying operators that will be useful in the sequel.
    We consider the Hilbert space $H=L^2((0,\infty),dt/t)$. Suppose that
    $(e_k)_{k=1}^\infty$ is an orthonormal basis in $H$ and $(\gamma_k)_{k=1}^\infty$ is a sequence of independent standard Gaussian random variables on a    probability space $(\Omega,\mathbb{P})$.
    A bounded operator $T$ from $H$ into $\B$ is a $\gamma-$radonifying operator, shortly
    $T \in \gamma(H,\B),$ when $\sum_{k=1}^\infty \gamma_k T e_k$ converges in $L^2(\Omega,\B).$ We define the norm
    $\|T\|_{\gamma(H,\B)}$ by
    $$ \|T\|_{\gamma(H,\B)}= \left( \mathbb{E}\left\|\sum_{k=1}^\infty \gamma_k Te_k\right\|_\B^2\right)^{1/2}.$$
    This definition does not depend on the orthonormal basis $(e_k)_{k=1}^\infty$ of $H.$ $\gamma(H,\B)$ is a
    Banach space which is continuously   contained in the space $L(H,\B)$ of bounded operators from $H$ into $\B.$\\

    If $f:(0,\infty)\longrightarrow \B$ is a measurable function such that for every $S \in \B^*,$ the dual space of
    $\B,$ $S \circ f\in H,$ there exists $T_f \in L(H,\B)$ for which
    $$\langle S, T_f(h)\rangle_{\B^*,\B}=\int_0^\infty \langle S,f(t)\rangle_{\B^*,\B} h(t) \frac{dt}{t}, \quad h\in H \text{ and } S \in \B^*,$$
    where $\langle \cdot, \cdot \rangle_{\B^*,\B}$ denotes the duality pairing in $(\B^*,\B)$.
    When $T_f\in \gamma(H,\B)$ we say that $f\in \gamma(H,\B)$ and we write $\|f\|_{\gamma(H,\B)}$ to refer us to
    $\| T_f\|_{\gamma(H,\B)}.$\\

    The Hilbert transform $\mathcal{H}(f)$ of $f\in L^p(\mathbb{R}),$ $1\leq p <\infty,$ is defined by
    $$ \mathcal{H}(f)(x)=\frac{1}{\pi}\lim_{\varepsilon \rightarrow 0^+ } \int_{|x-y|>\varepsilon } \frac{f(y)}{x-y}, \quad \text{a.e. } x\in \mathbb{R}.$$
    The Hilbert transform $\mathcal{H}$ is defined on
    $L^p(\mathbb{R})\otimes \mathbb{B},$ $1\leq p<\infty$, in a natural way. We say that $\mathbb{B}$ is a UMD
    Banach space when for some (equivalent, for every) $1<p<\infty$ the Hilbert transformation can be extended from
    $L^p(\mathbb{R}, \mathbb{B})$ as a bounded operator from $L^p(\mathbb{R}, \mathbb{B})$ into itself.
    There exist many other characterizations of the UMD Banach spaces
    (see, for instance, \cite{AT}, \cite{Bou}, \cite{Bu}, \cite{GD}, \cite{HTV}, \cite{Hy}  and \cite{KW1}).
    Every Hilbert space is a UMD space and $\gamma(H,\B)$ is UMD provided that $\B$ is UMD.\\

    UMD Banach spaces are a suitable setting to establish Banach valued Fourier multiplier theorems
    (\cite{GW} and \cite{Hy1}). Convolution operators are closely connected with Fourier multipliers.
    Suppose that $\psi\in L^2(\mathbb{R}^n).$ We consider $\psi_t(x) =\frac{1}{t^n}\psi(x/t)$, $x\in \R$
    and $t>0.$ The wavelet transform $W_\psi$ associated with $\psi$ is defined by
    $$ W_\psi(f)(x,t)=(f* \psi_t)(x),\quad x\in \R \text{ and } t>0,$$
    where $f\in \mathcal{S}(\R,\B),$ the $\B$-valued Schwartz space.\\

    In \cite[Theorem 4.2]{KW1} Kaiser and Weis gave sufficient conditions for $\psi$ in order to
    \begin{equation}\label{1.3}
    \|W_\psi f\|_{E(\R,\gamma(H,\B))}\sim \|f\|_{E(\R,\B)},
    \end{equation}
    for every $f\in E(\R,\mathbb{B})$, where $\mathbb{B}$ is a UMD Banach space and $E$ represents
    $L^p,$ $1<p<\infty,$ $H^1$ or $BMO.$ Here, as usual, $H^1$ and $BMO$ denote the Hardy spaces and the space of
    bounded mean oscillation functions, respectively.\\

    If $P(x)= \G((n+1)/2)/\pi^{(n+1)/2} (1+|x|^2)^{-(n+1)/2},$ $ x\in \R,$ then
    $P_t(x)=\frac{1}{t^n}P(\frac{x}{t}),$ $x\in \R$ and $t>0,$ is the classical Poisson kernel. By taking
    $ \psi(x)=\partial_t P_t(x)_{|t=1},$ $x\in \R,$ we have that
    $$ W_\psi(f)(x,t)=t\partial_t P_t(f)(x), \quad x\in \R \text{ and } t>0.$$
    Moreover, $\gamma(H,\mathbb{C})=H$ and $\gamma(H,\mathbb{H})=L^2((0,\infty), dt/t; \mathbb{H}),$ provided
    that $\mathbb{H}$ is a Hilbert space \cite[p. 3]{Nee}. Then, when $E=L^p,$ $1<p<\infty,$ \eqref{1.3} can be
    seen as a Banach valued extension of \eqref{1.1} and \eqref{1.2}.\\

    Also, in \cite[Remark 4.6]{KW1} UMD Banach spaces are characterized by using wavelet transforms.\\

    Harmonic analysis associated with the harmonic oscillator (also called Hermite) operator $L=-\Delta+|x|^2$ on
    $\R$ has been developed in last years by several authors
    (see \cite{AT}, \cite{BFRST1}, \cite{ST1}, \cite{ST2}, \cite{ST3}, \cite{Th}   and  \cite{Th1}, amongst others).
    Littlewood-Paley $g$-functions in the Hermite setting were analyzed in \cite{ST2} for scalar functions and in
    \cite{BFRST2} for Banach valued functions. Motivated by the ideas developed by Kaiser and Weis \cite{KW1},
    the authors in \cite[Theorem 1]{BCCFR} established new equivalent norms for the Bochner-Lebesgue space $L^p(\R,\B)$ by using
    Littlewood-Paley functions associated with Poisson semigroups for the Hermite operator and $\gamma-$radonifying
    operators, provided that $\B$ is a UMD space. Our objectives in this paper are the following ones:
    \begin{itemize}
    \item[$(a)$] To obtain equivalent norms for the $\B$-valued Hardy space $H^1_\mathcal{L}(\R,\B)$ and $BMO_\mathcal{L}(\R,\B)$
    associated to the Hermite operator, when $\B$ is a UMD Banach space, and
    \item[$(b)$] To characterize the UMD  Banach spaces in terms of $H_\mathcal{L}^1(\R,\B)$ and $BMO_{\mathcal{L}}( \R,\B),$
    by using Littlewood-Paley functions for the Poisson semigroup in the Hermite context and $\gamma-$radoni\-fying
    operators.\\
    \end{itemize}

    We recall some definitions and properties about the Hermite setting. For every $k\in \mathbb{N}$ the $k$-th
    Hermite function is $h_k(x)=(\sqrt{\pi}2^k k!)^{-1/2}H_k(x)e^{-x^2/2},$ $x\in \mathbb{R},$ where $H_k$ represents
    the $k$-th Hermite polynomial \cite[p. 60]{Leb}. If $k=(k_1,\dots,k_n)\in \mathbb{N}^n$ the $k$-th multidimensional
    Hermite function $h_k$ is defined by
    $$ h_k(x)=\prod_{j=1}^n h_{k_j}(x_j), \quad x=(x_1,\dots,x_n)\in \R,$$
    and we have that
    $$Lh_k=(2|k|+n)h_k,$$
where $|k|=k_1+...+k_n$.
    The system $\{ h_k\}_{k\in \mathbb{N}^n}$ is a complete orthonormal system for $L^2(\R).$ We define,
    the operator $\mathcal{L}$ as follows
    $$ \mathcal{L} f=\sum_{k\in \mathbb{N}^n} (2|k|+n)\langle f, h_k \rangle h_k, \quad f\in D(\mathcal{L}),$$
    where the domain $D(\mathcal{L})$ is constituted by all those $f\in L^2(\R)$ such that
    $\sum_{k\in \mathbb{N}^n}(2|k|+n)^2 |\langle f,h_k\rangle |^2 <\infty.$ Here $\langle \cdot, \cdot \rangle$ denotes
    the usual inner product in $L^2(\R).$ It is clear that if $\phi\in C_c^\infty(\R),$
    the space of smooth functions with compact support in $\R$, then $L\phi=\mathcal{L}\phi.$\\

    For every $t>0$ we consider the operator $W_t^\mathcal{L}$ defined by
    $$ W_t^\mathcal{L}(f)=\sum_{k\in \mathbb{N}^n} e^{-t(2|k|+n)}\langle f, h_k\rangle h_k, \quad f\in L^2(\R).$$
    The family $\{ W_t^\mathcal{L}\}_{t>0}$ is a semigroup of operators generated by $-\mathcal{L}$ in
    $L^2(\R)$ which is usually called the heat semigroup associated to $\mathcal{L}.$ By taking into account
    the Mehler's formula \cite[(1.1.36)]{Th} we can write, for every $f\in L^2(\R),$
    $$ W_t^\mathcal{L}(f)(x)=\int_{\R} W_t^\mathcal{L}(x,y)f(y)dy, \quad x\in \R \text{ and } t>0,$$
    where, for every $x,y\in \R$ and $t>0$,
    $$ W_t^\mathcal{L}(x,y)
        = \left(\frac{e^{-2t}}{\pi(1-e^{-4t})}\right)^{n/2}\exp\left(-\frac{1}{4}\left(\frac{1+e^{-2t}}{1-e^{-2t}}|x-y|^2+ \frac{1-e^{-2t}}{1+e^{-2t}}|x+y|^2 \right) \right).$$
    The Poisson semigroup $\{ P_t^\mathcal{L}\}_{t>0}$ associated to $\mathcal{L}$,
    that is, the semigroup of operators generated by $-\sqrt{\mathcal{L}},$ can be written by using the subordination formula by
    \begin{equation}\label{1.3.1}
        P_t^\mathcal{L}(f)= \frac{t}{\sqrt{4\pi}}\int_0^\infty s^{-3/2}e^{-t^2/4s}W_s^\mathcal{L}(f)ds, \quad f\in L^2(\R) \text{ and } t>0.
    \end{equation}
    The families $\{ W_t^\LL\}_{t>0}$ and $\{ P_t^\mathcal{L}\}_{t>0}$ are also $C_0-$semigroups in $L^p(\R),$ for
    every $1<p<\infty$ (see \cite{Ste}), but they are not Markovian.\\

    In \cite{ST2} Stempak and Torrea studied the Littlewood Paley $g$-functions in the Hermite setting.
    They proved that the $g$-function defined by
    $$ g(\{ P_t^\LL\}_{t>0})(f)(x)=\left( \int_0^\infty |t\partial_t P_t^\LL f(x)|^2 \frac{dt}{t}\right)^{1/2}, \quad x \in \R,$$
    is bounded from $L^p(\R)$ into itself, when $1<p<\infty$ (\cite[Theorem 3.2]{ST2}). Also, we have that
    \begin{equation}\label{1.4}
    \|f\|_{L^p(\R)}\sim \|g(\{P_t^\LL \}_{t>0})\|_{L^p(\R)},\quad f\in L^p(\R),
    \end{equation}
    (\cite[Proposition 2.3]{BFRTT}).\\

    From \cite[Theorems 1 and 2]{BFRST2} and \cite{Kw} we deduce that by defining, for every $1<p<\infty,$
    $$ g_\B(\{ P_t^\LL\}_{t>0})(f)(x)=\left(\int_0^\infty \| t\partial_t P_t^\LL f(x)\|_\B^2 \frac{dt}{t} \right)^{1/2}, \quad f\in L^p(\R,\B),$$
    then, for some (equivalently, for every) $1<p<\infty,$
    $$ \|f\|_{L^p(\R,\B)}\sim \|g_\B(\{ P_t^\LL\}_{t>0})(f) \|_{L^p(\R)}, \quad f\in L^p(\R,\B),$$
    if, and only if, $\B$ is isomorphic to a Hilbert space.\\

    We consider the operator $\mathcal{G}_{\LL,\B}$ defined by
    $$\mathcal{G}_{\LL,\B} (f)(x,t)=t\partial_t P_t^\LL (f)(x), \quad x\in \R \text{ and } t>0, $$
    for every $f\in L^p(\R,\B),$ $1\leq p<\infty.$\\

    In \cite{BCCFR} the authors proved that, for every $1<p<\infty,$
    \begin{equation}\label{1.5}
    \|f\|_{L^p(\R,\B)} \sim \| \mathcal{G}_{\LL,\B}(f)\|_{L^p(\R, \gamma(H,\B))},
    \end{equation}
    provided that $\B$ is a UMD Banach space. Since $\gamma(H,\mathbb{C})=H$, \eqref{1.5} can be seen as a
    Banach valued extension of \eqref{1.4}.\\

    Our first objective is to establish \eqref{1.5} when the space $L^p$ is replaced by the Hardy space $H^1$
    and the BMO space associated with the Hermite operator.\\

    Dziuba\'nski and Zienkiewicz \cite{DZ1} investigated the Hardy space $H^1_{\mathcal{S}_V}(\R)$ in the Schr�dinger context, where
    $\mathcal{S}_V=-\Delta + V$ and $V$ is a suitable positive potential.
    The Hermite operator is a special case of the Schr\"odinger operator.
    In \cite{DGMTZ} the dual space of $H^1_{\mathcal{S}_V}(\R)$ is characterized as the space $BMO_{\mathcal{S}_V}(\R)$ that is contained in the
    classical $BMO(\R)$ of bounded mean oscillation function in $\R.$ The results in \cite{DGMTZ} and \cite{DZ1} hold
    when the dimension $n$ is greater than 2, but when $V(x)=|x|^2$, $x\in \R,$ that is, when $\mathcal{S}_V=\LL$ the results in
    \cite{DGMTZ} and \cite{DZ1} about Hardy and BMO spaces hold for every dimension $n\geq 1.$\\

    We say that a function $f\in L^1(\R,\B)$ is in $H^1_\LL(\R,\B)$ when
    $$ \sup_{t>0}\|  W_t^\LL (f)\|_\B \in L^1(\R).$$
    As usual we consider on $H_\LL^1 (\R,\B)$ the norm $\| \cdot \|_{H^1_\LL(\R,\B)}$ defined by
    $$\| f\|_{H^1_\LL(\R,\B)}= \|\sup_{t>0}\|W_t^\LL(f) \|_\B \|_{L^1(\R)}, \quad f\in H^1_\LL(\R,\B).$$
    The dual space of $H_\LL^1(\R,\B)$ is the space $BMO_\LL(\R,\B^*)$ defined as follows, provided that $\B$
    satisfies the Radon-Nikod\'ym property (see \cite{Bl}). Note that every UMD space is reflexive (\cite[Proposition 2, p. 205]{RdF}) and therefore verifies
    the Radon-Nikod\'ym property (\cite[Corollary 13, p. 76]{DU}).
    A function $f\in L^1_{\text{loc}}(\R,\B)$ is in $BMO_\LL(\R,\B)$ if there exists $C>0$ such that
    \begin{itemize}
    \item[$(i)$] for every $a\in \R$ and $0<r<\rho(a)$
    $$ \frac{1}{| B(a,r)|}\int_{B(a,r)}\|f(z)-f_{B(a,r)}\|_B dz \leq C,$$
    where $f_{B(a,r)}= \frac{1}{B(a,r)}\int_{B(a,r)}f(z)dz,$ and
    \item[$(ii)$] for every $a\in\R$ and $r\geq \rho(a),$
    $$ \frac{1}{|B(a,r)|}\int_{B(a,r)}\|f(z)\|_\B dz \leq C.$$
    \end{itemize}
    Here $\rho$ is given by
    $$ \rho(x)=\left\{
        \begin{array}{ll}
            \dfrac{1}{1+|x|}, & |x|\geq 1\\
            \\
            \dfrac{1}{2}, &|x|<1
        \end{array}\right. .$$

    When $\B=\mathbb{C}$ we simply write $H^1_\LL(\R)$ and $BMO_\LL(\R)$, instead of $H^1_\LL(\R,\mathbb{C})$ or $BMO_\LL(\R,\mathbb{C})$,
    respectively.\\

    In \cite{BCFST} it was established a T1 type theorem that gives sufficient conditions in order that an operator
    is bounded between $BMO_\LL$ spaces.\\

    Suppose that $\B_1$ and $\B_2$ are Banach spaces and $T$ is a linear operator bounded from $L^2(\R, \B_1)$ into
    $L^2(\R,\B_2)$ such that
    $$ T(f)(x)=\int_\R K(x,y) f(y) dy, \quad x\not \in \supp f, \quad f\in L^\infty_c(\R, \B_1),$$
    where $K(x,y)$ is a bounded operator from $\B_1$ into $\B_2,$ for every $x,y\in \R,$ $x\neq y,$ and the
    integral is understood  in the $\B_2-$Bochner sense.\\

    As in \cite{BCFST} we say that $T$ is a $(\B_1,\B_2)$-Hermite-Calder\'on-Zygmund operator when the following two conditions
    are satisfied:
    \begin{itemize}
    \item[$(i)$] $\displaystyle \| K(x,y)\|_{L(\B_1, \B_2)}\leq C \frac{e^{-c(|x-y|^2+ |x||x-y|)}}{|x-y|^n}, \quad x,y\in \R, \ x\neq y,$
    \item[$(ii)$] $\displaystyle \| K(x,y)-K(x,z)\|_{L(\B_1, \B_2)}+ \| K(y,x)-K(z,x) \|_{L(\B_1, \B_2)}\leq C \frac{|y-z|}{|x-y|^{n+1}}, \quad |x-y|>2|y-z|$,
    \end{itemize}
    where $C,c>0$ and $L(\B_1, \B_2)$ denotes the space of bounded operators from $\B_1$ into $\B_2.$\\

    If $T$ is a Hermite-Calder\'on-Zygmund operator, we define the operator $\mathbb{T}$ on $BMO_\LL(\R, \B_1)$ as
    follows: for every $f\in BMO_\LL(\R, \B_1),$
    $$ \mathbb{T}(f)(x)
        =T(f\chi_{B})(x)+ \int_{\mathbb{R}^n \setminus B} K(x,y)f(y)dy,
        \quad \text{ a.e. } x\in B=B(x_0,r_0), \ x_0 \in \mathbb{R}^n \text{ and } r_0>0.$$
    This definition is consistent in the sense that it does not depend on $x_0$ or $r_0.$ Note that if $f\in BMO_\LL (\R, \B_1)$,
    $B=B(x_0,r_0),$ and $B^*=B(x_0,2r_0)$ where $x_0\in \R$ and $r_0>0,$ then
    $$ \mathbb{T} (f)(x)=T((f-f_B)\chi_{B^*})(x)+ \int_{\R \setminus B^*} K(x,y)(f(y)-f_B)dy +  \mathbb{T}(f_B)(x), \text{ a.e }x\in B^*.$$
    Note that if $f \in L^\infty_c(\R,\B_1)$ then $\mathbb{T}(f)=T(f)$.
    In Theorems \ref{Th1.1} and \ref{Th1.2} below we establish the boundedness of certain Banach valued
    Hermite-Calder\'on-Zygmund operators between $BMO_\LL$ spaces. When we say that an operator $T$ is bounded between
    $BMO_\LL$ spaces we always are speaking of the corresponding operator $\mathbb{T}$, although we continue writing
    $T.$ In order to show the boundedness of our operators in Banach valued $BMO_\LL$ spaces
    we will use a Banach valued version of \cite[Theorem 1.1]{BCFST} (see \cite[Remark 1.1]{BCFST}).

    \begin{Th}\label{Th1.0}
        Let $\B_1$ and $\B_2$ be Banach spaces. Suppose that $T$ is a $(\B_1,\B_2)$ Hermite-Calder\'on-Zygmund
        operator. Then, the operator $T$ is bounded from $BMO_\LL(\R,\B_1)$ into $BMO_\LL(\R,\B_2)$
        provided that there exists $C>0$ such that:
        \begin{itemize}
            \item[$(i)$] for every $b \in \B_1$ and $x \in \R$,
            $$\frac{1}{|B(x,\rho(x))|} \int_{B(x,\rho(x))} \|T(b)(y)\|_{\B_2} dy \leq C \|b\|_{\B_1},$$
            \item[$(ii)$] for every $b \in \B_1$, $x \in \R$ and $0<s \leq \rho(x)$,
            $$\left( 1 + \log \left( \frac{\rho(x)}{s} \right) \right)\frac{1}{|B(x,s)|}
                \int_{B(x,s)} \|T(b)(y) - (T(b))_{B(x,s)}\|_{\B_2} dy \leq C \|b\|_{\B_1},$$
            where $ \displaystyle (T(b))_{B(x,s)} = \frac{1}{|B(x,s)|} \int_{B(x,s)} T(b)(y) dy $.
        \end{itemize}
    \end{Th}

    This result can be proved in the same way as \cite[Theorem 1.1]{BCFST}. In some special cases the conditions
    $(i)$ and $(ii)$  reduce to simpler forms. For instance, if $T(b)=\widetilde{T}(1)b$, $b \in \B_1$, where
    $\widetilde{T}$ is a $(\mathbb{C}, L(\B_1,\B_2))$ operator (where $(\mathbb{C}, L(\B_1,\B_2))$ has the
    obvious meaning) then properties $(i)$ and $(ii)$ are satisfied provided that $\widetilde{T}(1) \in L^\infty(\R,L(\B_1,\B_2))$
    and $\nabla \widetilde{T}(1) \in L^\infty(\R, L(\B_1,\B_2))$.\\

    We denote by $\{ P_t^{\LL + \alpha}\}_{t>0}$ the Poisson semigroup associated with the operator $\LL+\alpha,$
    when $\alpha>-n.$ We can write
    $$ P_t^{\LL+ \alpha}(f)= \frac{t}{\sqrt{4\pi}}\int_0^\infty s^{-3/2}e^{-t^2/(4s)}e^{-\alpha s}W_s^\LL(f)ds.$$
    The operator $\mathcal{G_{\LL+\alpha,\B}}$ is defined by
    $$\mathcal{G}_{\LL+\alpha, \B}(f)(x,t)=t\partial_t P_t^{\LL+\alpha}(f)(x),\quad x\in \mathbb{R}^n \mbox{ and }t>0.$$

    Our first result is the following one.

    \begin{Th}\label{Th1.1}
        Let $\B$ be a UMD Banach space and $\alpha>-n$. Then, if $E$ represents $H^1_\LL$ or $BMO_\LL$ we have that
        $$ \|f\|_{E(\R,\B)}\sim \| \mathcal{G}_{\LL+ \alpha,\B} (f)\|_{E(\R, \gamma(H,\B))}, \quad f\in E(\R,\B).$$
    \end{Th}

    In order to establish our characterization for the UMD Banach spaces we introduce the operators $T_{j,\pm}^\LL,$ $j=1, \dots, n,$ defined as follows:
    $$ T_{j,\pm}^\LL (f)(x,t)= t\left(\partial_{x_j} \pm x_j \right) P_t^\LL (f)(x), \quad x\in \mathbb{R}^n \mbox{ and }t>0.$$

    In \cite[Theorem 2]{BCCFR} it was established that if $\B$ is a UMD Banach space then the operators
    $T_{j,\pm}^\LL$ are bounded from $L^p(\R,\B)$ into $L^p(\R,\gamma(H,\B))$, for every
    $1<p<\infty$ and $j=1, \dots, n$, provided that $n \geq 3$ in the case of $T_{j,-}^\LL$.\\

    The behavior of the operators $T_{j,\pm}^\LL$ on the spaces $H^1_\LL(\R,\B)$ and $BMO_\LL(\R,\B)$
    is now stated.

    \begin{Th}\label{Th1.2}
        Let $\B$ be a UMD Banach space and $j=1,\dots, n.$ By $E$ we represent the space $H^1_\LL$ or $BMO_\LL.$
        Then, the operators $T_{j,\pm}^\LL$ are bounded from $E(\R, \B)$ into $E(\R, \gamma(H, \B)),$
        provided that $n \geq 3$ in the case of  $T_{j,-}^\LL$.
    \end{Th}

    UMD Banach spaces are characterized as follows.
    \begin{Th}\label{Th1.3}
    Let $\B$ be a Banach space. Then, the following assertions are equivalent.
    \begin{itemize}
    \item[$(i)$] $\B$ is UMD.
    \item[$(ii)$] For some (equivalently, for every) $j=1, \dots,n,$ there exists $C>0$ such that, for every $f\in H_\LL^1(\R) \otimes \B,$
        $$ \|f\|_{H^1_\LL(\R, \B)} \leq C\| \mathcal{G}_{\LL+2,\B}(f)\|_{H^1_\LL(\R, \gamma(H,\B))},$$
        and
        $$ \|T_{j,+}^\LL (f)\|_{H^1_\LL(\R,\gamma(H,\B))} \leq C\|f\|_{H_\LL^1 (\R, \B)}. $$
    \item[$(iii)$] For some (equivalently, for every) $j=1,\dots, n,$ there exists $C>0$ such that, for every
    $f\in L^\infty_c(\R) \otimes \B,$
    $$ \|f\|_{BMO_\LL(\R,\B)}  \leq C\| \mathcal{G}_{\LL+2,\B} (f) \|_{BMO_\LL (\R, \gamma(H,\B))}$$
    and
    $$ \|T^\LL_{j,+} (f)\|_{BMO_\LL (\R, \gamma(H,\B))}\leq C \|f\|_{BMO_\LL(\R, \B)}.$$
    \end{itemize}
    In $(ii)$ and $(iii)$ the operators $\mathcal{G}_{\LL+2, \B}$ and $T_{j,+}^\LL,$ $j=1, \dots, n,$ can be replaced by
    $\mathcal{G}_{\LL-2, \B}$ and $T_{j,-}^\LL,$ $j=1, \dots, n,$ respectively, provided that $n\geq 3.$
    \end{Th}

    In the following sections we present proofs of Theorems \ref{Th1.1}, \ref{Th1.2} and \ref{Th1.3}.
    In the Appendix (Section~\ref{sec:app}) we show that the Riesz transforms in the Hermite setting can be extended as bounded
    operators from $BMO_\LL(\R, \B)$ into itself and from  $H^1_\LL(\R, \B)$ into itself. These boundedness properties will be needed when proving
    Theorem~\ref{Th1.3}. Moreover, they have interest in themself  and complete the results established
    in \cite{BCFST} and in \cite{DZ1}.\\

    Throughout this paper by $C$ and $c$ we always denote positive constants that can change on each occurrence.

%\newpage
    %%%%%%%%%%%%%%%%%%%%%%%%%%%%%%%%%%%%%%%%%%%%%%%%%%%%%%%%%%%%%%%%%%%%%%%%%%%%%%%%%%%%%%%%%%%%%%%%%%%%%%%%%%%%%%%%%%%%
    \section{Proof of Theorem~\ref{Th1.1}}  \label{sec:Th1.1}
    %%%%%%%%%%%%%%%%%%%%%%%%%%%%%%%%%%%%%%%%%%%%%%%%%%%%%%%%%%%%%%%%%%%%%%%%%%%%%%%%%%%%%%%%%%%%%%%%%%%%%%%%%%%%%%%%%%%%

    We distinguish four parts in the proof of Theorem~\ref{Th1.1}.

%%%%%%%%%%%%%%%%%
    \subsection{}\label{subsec:2.1}

    We are going to show that the operator $\mathcal{G}_{\mathcal{L}+\alpha, \B}$ is bounded from
    $BMO_\mathcal{L}(\R,\B)$ into $BMO_\mathcal{L}(\R,$ $\gamma(H,\B))$. In order to see this
    we will use Theorem~\ref{Th1.0}. According to \cite[Theorem 1]{BCCFR} the operator
    $\mathcal{G}_{\mathcal{L}+\alpha, \B}$ is bounded from $L^2(\R,\B)$ into $L^2(\R,\gamma(H,\B))$,
    because $\B$ is UMD. \\

    Suppose that $f \in BMO_\mathcal{L}(\R,\B)$. Then, $f$ is a $\B$-valued function with bounded mean
    oscillation and hence $\int_{\R} \|f(x)\|_\B/(1+|x|)^{n+1} dx < \infty$. The kernel
    $P_t^{\mathcal{L}+\alpha}(x,y)$ of the operator $P_t^{\mathcal{L}+\alpha}$ can be written as
    $$P_t^{\mathcal{L}+\alpha}(x,y)
        = \frac{t}{\sqrt{4 \pi}} \int_0^\infty s^{-3/2} e^{-t^2/(4s)-\alpha s} W_s^\mathcal{L}(x,y) ds,
        \quad x,y \in \R \text{ and } t>0.$$
    We have that
    \begin{align*}
        t \partial_t P_t^{\mathcal{L}+\alpha}(x,y)
            = & \frac{t}{\sqrt{4 \pi}} \int_0^\infty s^{-3/2} \left(1-\frac{t^2}{2s}\right) e^{-t^2/(4s)-\alpha s} W_s^\mathcal{L}(x,y)ds,
              \quad x,y \in \R \text{ and } t>0.
    \end{align*}
 By \cite[(4.4) and (4.5)]{BCFST} we have that, for every $x,y \in \R \text{ and } s>0$,
 \begin{align}\label{2.1}
         W_s^\mathcal{L}(x,y)
            \leq &C\frac{e^{-ns}}{(1-e^{-4s})^{n/2}} \exp\left(-c\left( \frac{|x-y|^2}{1-e^{-2s}} + (1-e^{-2s})|x+y|^2+(|x|+|y|)|x-y|\right)\right)\nonumber\\
    \leq& Ce^{-c(|x-y|^2 + (|x|+|y|)|x-y|)}\frac{e^{-ns-c\frac{|x-y|^2}{s}-c(1-e^{-2s})|x+y|^2}}{(1-e^{-4s})^{n/2}}.
    \end{align}  Hence, since $\alpha +n>0$, for each $x,y \in \R \text{ and } t>0$,
  \begin{align}\label{2.0}
        | t \partial_t P_t^{\mathcal{L}+\alpha}(x,y)|
            \leq & C t e^{-c(|x-y|^2 + (|x|+|y|)|x-y|)} \int_0^\infty \frac{e^{-c \frac{|x-y|^2+t^2}{s}}}{s^{3/2}}\frac{e^{-(\alpha +n)s}}{(1-e^{-4s})^{n/2}}ds\nonumber\\
\leq &Ct e^{-c(|x-y|^2 + (|x|+|y|)|x-y|)}\int_0^\infty \frac{e^{-c \frac{|x-y|^2+t^2}{s}}}{s^{(n+3)/2}}ds
 \nonumber \\
            \leq & C e^{-c(|x-y|^2 + (|x|+|y|)|x-y|)}\frac{t}{(t+|x-y|)^{n+1}}\leq C\frac{t}{(t+|x-y|)^{n+1}}.
    \end{align}
 Then, $\int_{\R} |t \partial_t P_t^{\mathcal{L}+\alpha}(x,y)| \ \|f(y)\|_\B dy < \infty$, for every
    $x \in \R$ and $t>0$, and we deduce that
    $$t \partial_t P_t^{\mathcal{L}+\alpha}(f)(x)
        = \int_{\R} t \partial_t P_t^{\mathcal{L}+\alpha}(x,y) f(y) dy,\quad x\in \mathbb{R}^n\mbox{ and }t>0 .
$$

Moreover, by (\ref{2.0}) we get that,
    \begin{align}\label{2.0.1}
        \|t \partial_t P_t^{\mathcal{L}+\alpha}(x,y)\|_H\leq &Ce^{-c( |x-y|^2 + |y|  |x-y| )}\left(\int_0^\infty \frac{t}{(t+|x-y|)^{2(n+1)}}dt\right)^{1/2}\nonumber\\
\leq &C \frac{e^{-c( |x-y|^2 + |y|  |x-y| )}}{|x-y|^n},\quad x,y \in \R, \ x \neq y.
    \end{align}

    Let $x,y \in \R$, $x \neq y$. We write $F(x,y;t)=t \partial_t P_t^{\mathcal{L}+\alpha}(x,y)$, $t>0$.
    Since $F(x,y; \cdot) \in H$, for every $b \in \B$, the function
    $F_b(x,y;t)=F(x,y;t)b$, $t>0$, defines an element $\tilde{F}_b(x,y;\cdot) \in \gamma(H, \B)$ satisfying that
    \begin{align*}
        \langle S, \tilde{F}_b(x,y;\cdot)(h) \rangle_{\B^*,\B}
            = & \int_0^\infty \langle S, F_b(x,y;t) \rangle_{\B^*,\B} h(t) \frac{dt}{t} \\
            = & \langle S, b \rangle_{\B^*,\B} \int_0^\infty F(x,y;t)  h(t) \frac{dt}{t},
            \quad S \in \B^* \text{ and } h \in H.
    \end{align*}
    Then, for every $b \in \B$,
    $$\tilde{F}_b(x,y;\cdot)(h)
        = \left(\int_0^\infty F(x,y;t) h(t) \frac{dt}{t}\right) b,
        \quad h \in H.$$
    We consider the operator $\tau(x,y)(b)=\tilde{F}_b(x,y;\cdot)$, $b \in \B$. We have that
    \begin{align}\label{f1}
        \|\tau(x,y)(b)\|_{\gamma(H,\B)}
            = & \left( \E \left\| \sum_{k=1}^\infty \gamma_k \tilde{F}_b(x,y;\cdot)(e_k)  \right\|_\B^2 \right)^{1/2}
            =   \left( \E \left\| \sum_{k=1}^\infty \gamma_k \int_0^\infty F(x,y;t)e_k(t) \frac{dt}{t} b  \right\|_\B^2 \right)^{1/2} \nonumber \\
            = & \left( \E \left|  \sum_{k=1}^\infty \gamma_k \int_0^\infty F(x,y;t)e_k(t) \frac{dt}{t} \right|^2  \right)^{1/2} \left\| b \right\|_\B
            = \|F(x,y;\cdot)\|_H \|b\|_\B,
            \quad b \in \B.
    \end{align}
    Hence, if $L(\B, \gamma(H,\B))$ denotes the space of bounded operators from $\B$ into $\gamma(H, \B)$, we obtain
    \begin{equation*}
        \|\tau(x,y)\|_{L(\B, \gamma(H,\B))}
        \leq C  \frac{e^{-c( |x-y|^2 + |y|  |x-y| )}}{|x-y|^n}.
    \end{equation*}
Let $j=1, \dots, n$. We have that
$$
        \partial_{x_j}\left( t \partial_t P_t^{\mathcal{L}+\alpha}(x,y) \right)
            =  \frac{t}{\sqrt{4 \pi}} \int_0^\infty s^{-3/2} \left(1-\frac{t^2}{2s}\right) e^{-t^2/(4s)-\alpha s}\partial_{x_j}(W_s^\mathcal{L}(x,y))ds
              \quad x,y \in \R \text{ and } t>0.
$$
Since
$$
\partial_{x_j}(W_s^\mathcal{L}(x,y))=-\frac{1}{2}\left( \frac{1+e^{-2s}}{1-e^{-2s}}(x_j-y_j) + \frac{1-e^{-2s}}{1+e^{-2s}}(x_j+y_j)\right)W_s^\mathcal{L}(x,y), \quad x,y \in \R \text{ and } s>0,
$$
we obtain that
\begin{equation}\label{derivW}
|\partial_{x_j}(W_s^\mathcal{L}(x,y))|\leq Ce^{-c(|x-y|^2 + (|x|+|y|)|x-y|)}\frac{e^{-ns-c\frac{|x-y|^2}{s}}}{(1-e^{-4s})^{(n+1)/2}},\quad x,y \in \R \text{ and } s>0.
\end{equation}
By proceeding as above we get
    $$\|\partial_{x_j}\left( t \partial_t P_t^{\mathcal{L}+\alpha}(x,y) \right)\|_H
        \leq \frac{C}{|x-y|^{n+1}},
        \quad x,y \in \R, \ x \neq y,$$
    and then
    $$\|\partial_{x_j} \tau(x,y)\|_{L(\B, \gamma(H,\B))}
        \leq \frac{C}{|x-y|^{n+1}},
        \quad x,y \in \R, \ x \neq y.$$
    By taking into account symmetries we obtain the same estimates when $\partial_{x_j}$ is replaced by $\partial_{y_j}$.\\

    Next we show that if $f \in L_c^\infty(\R,\B)$ then
    \begin{equation}\label{A}
        t \partial_t P_t^{\mathcal{L}+\alpha}(f)(x)
            =  \int_{\R} \tau(x,y) f(y)dy,
        \quad x \notin \supp(f),
    \end{equation}
    where the integral is understood in the $\gamma(H,\B)$-Bochner sense. Indeed, let $f \in L_c^\infty(\R,\B)$ and
    $x \notin \supp(f)$. We have that
    $$\int_{\R} \|  \tau(x,y) f(y)\|_{\gamma(H,\B)} dy
        \leq C \int_{\supp(f)} \frac{\|f(y)\|_\B}{|x-y|^n} dy
        < \infty.$$
    Since $\gamma(H,\B)$ is continuously contained in the space $L(H,\B)$,
    $ \tau(x,\cdot) f \in L^1(\R,L(H,\B))$. Then, there exists a sequence $(T_k)_{k \in \mathbb{N}}$
    in $L^1(\R) \otimes L(H,\B)$ such that
    $$T_k
        \longrightarrow \tau(x,\cdot) f,
        \quad \text{as } k \to \infty, \text{ in } L^1(\R,L(H,\B)).$$
    Hence,
    $$\int_{\R} T_k(y) dy
        \longrightarrow \int_{\R} \tau(x,y) f(y) dy,
        \quad \text{as } k \to \infty, \text{ in } L(H,\B), $$
    and also, for every $h \in H$,
    $$T_k[h]
        \longrightarrow \tau(x,\cdot)f[h],
        \quad \text{as } k \to \infty, \text{ in } L^1(\R,\B).$$
    Suppose that $T=\sum_{\ell=1}^m f_{\ell} \tau_\ell$, where $f_\ell \in L^1(\R)$ and $\tau_\ell \in L(H,\B)$,
    $\ell =1, \dots, m \in \mathbb{N}$. We can write
    $$\left( \int_{\R} T(y) dy \right)[h]
        = \sum_{\ell=1}^m  \tau_\ell[h] \int_{\R} f_{\ell}(y) dy
        =  \int_{\R} T(y)[h] dy,
        \quad h \in H.$$
    Hence, we conclude that
    $$ \left(\int_{\R} \tau(x,y) f(y) dy\right)[h]
        =  \int_{\R} \tau(x,y) f(y)[h] dy,
        \quad h \in H,$$
    where the last integral is understood in the $\B$-Bochner sense.\\

    For every $h \in H$, by \eqref{2.0.1} we have that
    \begin{align*}
        \int_{\R} \tau(x,y) f(y)[h] dy
            = & \int_{\supp(f)}  \left(\int_0^\infty t \partial_t P_t^{\mathcal{L}+\alpha}(x,y) h(t) \frac{dt}{t}\right) f(y) dy \\
            = & \int_0^\infty \left(\int_{\supp(f)} t \partial_t P_t^{\mathcal{L}+\alpha}(x,y) f(y) dy\right) h(t) \frac{ dt}{t} \\
            = & \int_0^\infty t \partial_t P_t^{\mathcal{L}+\alpha}(f)(x)  h(t) \frac{dt}{t}
            =  \left(t \partial_t P_t^{\mathcal{L}+\alpha}(f)(x)\right)[h].
    \end{align*}
    Thus \eqref{A} is established.\\

    We conclude that $\mathcal{G}_{\mathcal{L}+\alpha,\B}$ is a $(\B,\gamma(H,\B))$-Hermite-Calder\'on-Zygmund
    operator.\\

    On the other hand, by \cite[Proposition 3.3]{ST4} we have that
    \begin{equation}\label{2.1.0}
        W_t^\mathcal{L}(1)(x)
            = \frac{1}{\pi^{n/2}} \left( \frac{e^{-2t}}{1+e^{-4t}} \right)^{n/2}
                \exp\left(-\frac{1-e^{-4t}}{2(1+e^{-4t})}|x|^2\right),
            \quad x \in \R \text{ and } t>0.
    \end{equation}
    It follows that, for every $x \in \R \text{ and } t>0$,
    \begin{equation}\label{2.1.1}
        \partial_t W_t^{\mathcal{L}+\alpha}(1)(x)
            = \partial_t \left( e^{-\alpha t} W_t^\mathcal{L}(1)(x) \right)= -e^{-\alpha t}\left(\alpha + n \frac{1-e^{-4t}}{1+e^{-4t}}+|x|^2 \frac{4e^{-4t}}{(1+e^{-4t})^2}\right) W_t^\mathcal{L}(1)(x).
    \end{equation}
    We can write
    \begin{align}\label{2.1.2}
        \mathcal{G}_{\mathcal{L}+\alpha,\mathbb{C}}(1)(x,t)
            = & \frac{t}{\sqrt{\pi}} \int_0^\infty \frac{e^{-u}}{\sqrt{u}} \partial_t W_{t^2/4u}^{\mathcal{L}+\alpha}(1)(x) du \nonumber \\
            = & \frac{t^2}{\sqrt{4\pi}} \int_0^\infty \frac{e^{-u}}{u^{3/2}} \partial_z W_z^{\mathcal{L}+\alpha}(1)(x)_{|{ z=t^2/4u}} du,
            \ x \in \R \text{ and } t>0.
    \end{align}
    Minkowski's inequality leads to
    \begin{align*}
        \|\mathcal{G}_{\mathcal{L}+\alpha,\mathbb{C}}(1)(x,\cdot)\|_H
            \leq & C \int_0^\infty \frac{e^{-u}}{u^{3/2}} \| t^2 \partial_z W_z^{\mathcal{L}+\alpha}(1)(x)_{|{ z=t^2/4u}} \|_H du \\
            \leq & C \int_0^\infty \frac{e^{-u}}{u^{1/2}} \| z \partial_z W_z^{\mathcal{L}+\alpha}(1)(x) \|_H du,
            \quad  x \in \R.
    \end{align*}
    Moreover, we have that
    \begin{align*}
        \| z \partial_z W_z^{\mathcal{L}+\alpha}(1)(x) \|_H
            \leq & C \left( \int_0^1 e^{-cz|x|^2}(1+|x|^4)z dz  + \int_1^\infty e^{-2(n+\alpha)z}z dz \right)^{1/2}
            \leq C, \quad  x \in \R.
    \end{align*}
    Hence, $\|\mathcal{G}_{\mathcal{L}+\alpha,\mathbb{C}}(1)(x,\cdot)\|_H \in L^\infty(\R)$.
    As above, this means that
    $\mathcal{G}_{\mathcal{L}+\alpha,\mathbb{C}}(1) \in L^\infty(\mathbb{R}^n,H)$.\\

    In a similar way we can see that, for every $j=1, \dots, n$,
    $\partial_{x_j}\mathcal{G}_{\mathcal{L}+\alpha,\mathbb{C}}(1) \in L^\infty(\mathbb{R}^n,H)$.\\

    By using Theorem~\ref{Th1.0} we can show that the operator
    $\mathcal{G}_{\mathcal{L}+\alpha,\B}$ is bounded from $BMO_\mathcal{L}(\R,\B)$ into
    $BMO_\mathcal{L}(\R,\gamma(H,\B))$.\\

%\newpage
%%%%%%%%%%%%%%%%%
    \subsection{}\label{subsec:2.2}

    We are going to prove that $\mathcal{G}_{\mathcal{L}+\alpha,\B}$ is a bounded operator from $H^1_\mathcal{L}(\R,\B)$
    into $H^1_\mathcal{L}(\R,\gamma(H,\B))$. In order to show this property we extend to a Banach valued setting the atomic characterization
    of Hardy spaces due to Dziuba\'nski and Zienkiewicz (\cite{Dz} and \cite{DZ1}).\\

    A strongly measurable function $a : \R \longrightarrow \B$ is an atom for $H^1_\mathcal{L}(\R,\B)$ when there
    exists $x_0 \in \R$ and $0<r_0 \leq \rho(x_0)$
    such that the support of $a$ is contained in $B(x_0,r_0)$ and
    \begin{itemize}
        \item[$(i)$]  $\displaystyle \|a\|_{L^\infty(\R,\B)} \leq |B(x_0,r_0)|^{-1}$,
        \item[$(ii)$] $\displaystyle \int_\R a(x)dx=0$, provided that $r_0 \leq \rho(x_0)/2$.
    \end{itemize}

    \begin{Prop}\label{Prop2.3}
        Let $Y$ be a Banach space. Suppose that $f \in L^1(\R,Y)$. The following assertions are equivalent.
        \begin{itemize}
            \item[$(i)$]   $\displaystyle \sup_{t>0} \|W_t^\mathcal{L}(f)\|_Y \in L^1(\R)$.
            \item[$(ii)$]  $\displaystyle \sup_{t>0} \|P_t^\mathcal{L}(f)\|_Y \in L^1(\R)$.
            \item[$(iii)$] There exists a sequence $(a_j)_{j \in \mathbb{N}}$ of atoms in $H^1_\mathcal{L}(\R,Y)$
            and a sequence $(\lambda_j)_{j \in \mathbb{N}}$ of complex numbers such that $\sum_{j\in \mathbb{N}}|\lambda_j| < \infty$
            and $f= \sum_{j\in \mathbb{N}}\lambda_j a_j$.
        \end{itemize}
    \end{Prop}

    \begin{proof}
        Dziuba\'nski and Zienkiewicz proved in \cite[Theorem 1.5]{DZ1} (see also \cite{Dz}) that
        $(i) \Leftrightarrow (iii)$ for $Y=\mathbb{C}$. In order to show \cite[Theorem 1.5]{DZ1} they use the atomic decomposition
        for the functions in the local Hardy space $h^1(\R)$ established by Goldberg (\cite[Lemma 5]{Go}). By reading carefully
        \cite[Theorem 1, p. 91,  and Theorem 2, p. 107]{Ste1} we can see that the classical Banach valued $H^1(\R,Y)$
        can be defined by using different maximal functions and by atomic representations, that is,
        \cite[Theorem 1, p. 91,  and Theorem 2, p. 107]{Ste1} continue being true when we replace $H^1(\R)$ by
        $H^1(\R,Y)$. Then, if we define the Banach valued local Hardy space $h^1(\R,Y)$ in the natural way, $h^1(\R,Y)$ can be
        described by the corresponding maximal functions and by atomic decompositions (see \cite[Theorem 1 and Lemma 5]{Go}). More precisely,
        the arguments in the proofs of \cite[Theorem 1 and Lemma 5]{Go} allow us to show that if $f \in L^1(\R,Y)$ then $f \in h^1(\R,Y)$
        if and only if $f=\sum_{j\in \mathbb{N}}\lambda_j a_j$, where $\lambda_j \in \mathbb{C}$, $j \in \mathbb{N}$, and $\sum_{j\in \mathbb{N}} |\lambda_j| < \infty$,
        and, for every $j \in \mathbb{N}$, $a_j$ is an $h^1$-atom as in \cite[p. 37]{Go} but taking values in $Y$. With these comments in mind and by proceeding
        as in the proof of \cite[Theorem 1.5]{DZ1} we conclude that $(i) \Leftrightarrow (iii)$.\\

        By the subordination representation \eqref{1.3.1} of $P_t^\mathcal{L}$, $t>0$,  we deduce that $(i) \Rightarrow (ii)$.\\

        To finish the proof we are going to see that $(ii) \Rightarrow (iii)$. In order to show this we can proceed as in the proof of
        \cite[Theorem 1.5]{DZ1}. We present a sketch of the proof. Firstly, by \eqref{1.3.1} and \eqref{2.1} and proceeding as in \eqref{2.0} we deduce that
        \begin{equation}\label{2.12}
            P_t^\mathcal{L}(x,y)
            \leq C e^{-c(|x-y|^2+|x||x-y|)}\frac{t}{\left( t + |x-y| \right)^{n+1}},
            \quad x,y \in \R \text{ and } t>0.
        \end{equation}
        Hence, for every $\ell \in \mathbb{N}$, there exists $C>0$ such that
        \begin{equation}\label{(A)}
            P_t^\mathcal{L}(x,y)
                \leq C \left( 1 + \frac{|x-y|}{\rho(x)}\right)^{-\ell}|x-y|^{-n},
            \quad x,y \in \R \text{ and } t>0.
        \end{equation}
        Moreover, for every $M>0$ we can find $C>0$ for which
        \begin{equation}\label{(B)}
            |P_t^\mathcal{L}(x,y)-P_t(x-y)|
                \leq C \left( \frac{|x-y|}{\rho(x)}\right)^{1/2}|x-y|^{-n},
            \quad x,y \in \R, \ |x-y| \leq M \rho(x) \text{ and } t>0,
        \end{equation}
        where $P_t$ denotes the classical Poisson semigroup.\\

        Indeed, let $M>0$.
        According to \eqref{1.3.1} we can write
        \begin{equation*}
            |P_t^\mathcal{L}(x,y)-P_t(x-y)|
                \leq C t \int_0^\infty \frac{e^{-t^2/4s}}{s^{3/2}} |W_s^\LL(x,y)-W_s(x-y)| ds,
            \quad x,y \in \R \text{ and } t>0,
        \end{equation*}
        where $W_t(x)=e^{-|x|^2/4t}/(4\pi t)^{n/2}$, $x \in \R$ and $t>0$. From \eqref{2.1} it follows that
        \begin{align*}
            & t \int_{\rho(x)^2}^\infty \frac{e^{-t^2/4s}}{s^{3/2}} |W_s^\LL(x,y)-W_s(x-y)| ds
                             \leq C t \int_{\rho(x)^2}^\infty \frac{e^{-c(t^2+|x-y|^2)/s}}{s^{(n+3)/2}} ds
              \leq C \int_{\rho(x)^2}^\infty \frac{ds}{s^{(n+2)/2}} \\
            & \qquad \qquad                 \leq \frac{C}{\rho(x)^n}
             = C \left( \frac{|x-y|}{\rho(x)} \right)^n \frac{1}{|x-y|^n},
            \quad x,y \in \R, \ x \neq y \text{ and } t>0.
        \end{align*}
        On the other hand, we have that
        \begin{align*}
            & t \int_0^{\rho(x)^2} \frac{e^{-t^2/4s}}{s^{3/2}} |W_s^\LL(x,y)-W_s(x-y)| ds
                             \leq C \Big\{ t \int_0^{\rho(x)^2} \frac{e^{-c(t^2+|x-y|^2)/s}}{s^{(n+3)/2}} |e^{-ns}-1| ds \\
            & \qquad \qquad \qquad +  t \int_0^{\rho(x)^2} \frac{e^{-c(t^2+|x-y|^2)/s}}{s^{3/2}}
                                        \left| \frac{1}{(1-e^{-4s})^{n/2}} - \frac{1}{(4s)^{n/2}}\right| ds \\
            & \qquad \qquad \qquad +  t \int_0^{\rho(x)^2} \frac{e^{-c(t^2+|x-y|^2)/s}}{s^{(n+3)/2}}
                                        \left| \exp\left( -\frac{1}{4}\frac{1-e^{-2s}}{1+e^{-2s}}|x+y|^2 \right)  - 1 \right| ds \\
            & \qquad \qquad \qquad +  t \int_0^{\rho(x)^2} \frac{e^{-t^2/4s}}{s^{(n+3)/2}}
                                        \left| \exp\left( -\frac{1}{4}\frac{1+e^{-2s}}{1-e^{-2s}}|x-y|^2 \right)  - e^{-|x-y|^2/4s} \right| ds \Big\} \\
            & \qquad \qquad = C \sum_{j=1}^4 I_j(x,y,t)
            \quad x,y \in \R \text{ and } t>0.
        \end{align*}
        Since $|e^{-ns}-1| \leq C s$, $s>0$, and
        $$ \left| \frac{1}{(1-e^{-4s})^{n/2}} - \frac{1}{(4s)^{n/2}}\right| \leq \frac{C}{s^{n/2-1}}, \quad 0<s<1,$$
       we deduce that
        \begin{align*}
            I_j(x,y,t)
                \leq & C  t \int_0^{\rho(x)^2} \frac{e^{-c(t^2+|x-y|^2)/s}}{s^{(n+1)/2}} ds
                \leq   C  \int_0^{1} \frac{e^{-c(t^2+|x-y|^2)/s}}{s^{n/2}} ds \leq   C  \frac{1}{(t^2+|x-y|^2)^{(n/2-1/4)}}\int_0^{1} \frac{ds}{s^{1/4}} \\
                \leq & \frac{C}{|x-y|^{n-1/2}}
                \leq C \left( \frac{|x-y|}{\rho(x)} \right)^{1/2} \frac{1}{|x-y|^n},
                \quad x,y \in \R, \ x \neq y, \ t>0 \text{ and } j=1,2.
        \end{align*}
        Also, we have that, for every $x,y \in \R$ and $s>0$,
        $$ \left| \exp\left( -\frac{1}{4}\frac{1+e^{-2s}}{1-e^{-2s}}|x-y|^2 \right)  - e^{-|x-y|^2/4s} \right|
            \leq C e^{-|x-y|^2/4s} |x-y|^2
            \leq C s e^{-c|x-y|^2/s}.$$
        Then, by proceeding as above we get
        $$I_4(x,y,t)
            \leq C \left( \frac{|x-y|}{\rho(x)} \right)^{1/2} \frac{1}{|x-y|^n},
                \quad x,y \in \R, \ x \neq y \text{ and } t >0.$$
        Finally, we analyze $I_3$. We have that
        $$\left| \exp\left( -\frac{1}{4}\frac{1-e^{-2s}}{1+e^{-2s}}|x+y|^2 \right)  - 1 \right|
            \leq C s |x+y|^2
            \leq C \frac{s}{\rho(x)^2},
            \quad |x-y| \leq M \rho(x) \text{ and } s>0.$$
     Hence, it follows that
        $$I_3(x,y,t)
            \leq \frac{C}{\rho(x)^2} \int_0^{\rho(x)^2} \frac{e^{-c(t^2+|x-y|^2)/s}}{s^{n/2}} ds\leq \frac{C}{\rho(x)^2|x-y|^{n-1/2}}\int_0^{\rho(x)^2} \frac{ds}{s^{1/4}}=C \left( \frac{|x-y|}{\rho(x)} \right)^{1/2} \frac{1}{|x-y|^n},
           $$
            provided that $|x-y| \leq M \rho(x)$, $x \neq y $ and $t>0$.

        By combining the above estimates we obtain \eqref{(B)}.\\

        Estimations \eqref{(A)} and \eqref{(B)} can be also obtained when $n \geq 3$ as special cases of
        \cite[Lemma 3.0]{DZ1}.\\

According to \cite[p. 517, line 5]{Shen}
$$
\rho (x)\sim \frac{1}{M(x)}=\sup\left\{r>0:\frac{1}{r^{n-2}}\int_{B(x,r)}|y|^2dy\leq 1\right\}.
$$
Since $\rho(x)\leq 1/2$, there exists $m_0\in \mathbb{Z}$ such that the set $\mathcal{B}_m=\{x\in \mathbb{R}^n: 2^{m/2}\leq M(x)<2^{\frac{m+1}{2}}\}$ is empty, provided that $m<m_0$.
Then, for every $m \in \mathbb{Z}$, $m\geq m_0$, and $k \in \mathbb{N}$ we can consider $x_{(m,k)} \in \R$
        as in \cite[Lemma 2.3]{DZ1} and choose, according to
        \cite[Lemma 2.5]{DZ1}, a function
        $\psi_{(m,k)} \in C_c^\infty(B(x_{(m,k)},2^{(2-m)/2}))$ such that
        $\| \nabla \psi_{(m,k)} \|_{L^\infty(\R)} \leq C 2^{m/2}$ and
        $\sum_{(m,k)} \psi_{(m,k)} = 1$, $x \in \R$. Here $C>0$ does not depend on $(m,k)$.
We can assume $m_0=0$ to make the reading easier.

For every $m, k\in \mathbb{N}$,
let us define $B_{(m,k)}=B(x_{(m,k)},2^{(4-m)/2})$
        and $\widehat{B}_{(m,k)}=B(x_{(m,k)},(\sqrt{n}+1)2^{(4-m)/2})$ and consider the maximal operators
        $$\widetilde{\mathcal{M}}_m (f) = \sup_{0<t\leq 2^{-m}} \| P_t (f) - P_t^\mathcal{L}(f) \|_Y, \quad
        \mathcal{M}_m^\mathcal{L}(f) = \sup_{0<t\leq 2^{-m}} \| P_t^\mathcal{L}(f) \|_Y, \quad
        \mathcal{M}_m (f) = \sup_{0<t\leq 2^{-m}} \| P_t (f) \|_Y, $$
        and the maximal commutator operator
        $$\mathcal{M}_{(m,k)}^\LL (f)
             = \sup_{0<t\leq 2^{-m}} \| P_t^\mathcal{L}(\psi_{(m,k)} f) - \psi_{(m,k)}P_t^\mathcal{L}(f) \|_Y.$$
        Let $m, k \in \mathbb{N}$.
        By using \eqref{2.12} we deduce that, for a certain $C>0$ independent of $m$ and $k$,
        $$\sup_{y \in B_{(m,k)}} \int_{\R \setminus \widehat{B}_{(m,k)}}
                \sup_{0<t\leq 2^{-m}}  |P_t^\mathcal{L} (x,y) -P_t(x,y)|dx
            \leq C. $$
        Indeed, if $x,y \in \R$, $x \neq y$, the function $w(t)=t/\left(t^2 + |x-y|^2\right)^{(n+1)/2}$,
        $t>0$, is increasing in the interval $(0,|x-y|/\sqrt{n})$ and it is decreasing in the interval
        $(|x-y|/\sqrt{n},\infty)$. If $x \in \R \setminus \widehat{B}_{(m,k)}$
        and $y \in B_{(m,k)}$, $|x-y| \geq \sqrt{n}2^{(4-m)/2}$. Hence, from \eqref{2.12} it follows that
        \begin{align*}
            & \sup_{y \in B_{(m,k)}} \int_{\R \setminus \widehat{B}_{(m,k)}} \sup_{0<t\leq 2^{-m}}  |P_t^\mathcal{L} (x,y)-P_t(x,y)| dx
                \leq C 2^{-m} \sup_{y \in B_{(m,k)}} \int_{\R \setminus \widehat{B}_{(m,k)}}
                    \frac{1}{\left( 2^{-2m} + |x-y|^2 \right)^{(n+1)/2}} dx \\
            & \qquad \qquad \leq C 2^{-m} \int_{\R \setminus B\left(0,\sqrt{n}2^{(4-m)/2}\right)}
                    \frac{1}{\left( 2^{-2m} + |u|^2 \right)^{(n+1)/2}} du
            \leq C \frac{2^{-m}}{2^{-m}+\sqrt{n}2^{(4-m)/2}}
            \leq C.
        \end{align*}

        By \eqref{(B)} and arguing as in \cite[Lemma 3.9]{DZ1}, we conclude that, for a certain $C>0$,
        $$\|\widetilde{\mathcal{M}}_m(\psi_{(m,k)} f)\|_{L^1(\R)}
            \leq C \| \psi_{(m,k)} f \|_{L^1(\R,Y)},
            \quad f \in L^1(\R,Y).$$

        Also, by proceeding as in the proof of \cite[Lemma 3.11]{DZ1} we can find $C>0$ such that
        $$\sum_{(m,k)} \|\mathcal{M}_{(m,k)}^\LL(f)\|_{L^1(\R)}
            \leq C \|  f \|_{L^1(\R,Y)},
            \quad f \in L^1(\R,Y).$$
        By combining the above estimates we deduce that
        $$\sum_{(m,k)} \|\mathcal{M}_m(\psi_{(m,k)}f)\|_{L^1(\R)}
            \leq C \left(\|  f \|_{L^1(\R,Y)} + \left\|\sup_{t>0} \| P_t^\mathcal{L} f \|_Y\right\|_{L^1(\R)} \right)
            < \infty,$$
        provided that $(ii)$ holds.\\

        Now the proof of $(ii) \Rightarrow (iii)$ can be finished as in \cite[Section 4]{DZ1}.
    \end{proof}

    In the next result we complete the last proposition characterizing the Hardy space by the maximal operator
    associated with the semigroup $\{P_t^{\mathcal{L} + \alpha}\}_{t>0}$.

    \begin{Prop}\label{Prop2.4}
        Let $Y$ be a Banach space and $\alpha > -n$. Suppose that $f \in L^1(\R,Y)$. Then $f \in H_\LL^1(\R,Y)$
        if, and only if, $\displaystyle \sup_{t>0} \|P_t^{\mathcal{L} + \alpha} (f)\|_Y \in L^1(\R)$.
    \end{Prop}

    \begin{proof}
        We consider the operator $L_\alpha$ defined by
        $$L_\alpha(g)
            = \sup_{t>0} \left\| P_t^{\mathcal{L} + \alpha}(g) - P_t^{\mathcal{L}}(g) \right\|_Y,
            \quad g \in L^1(\R,Y).$$
        We can write
        $$L_\alpha(g)(x)
            = \sup_{t>0} \left\| \int_\R L_\alpha(x,y;t) g(y) dy \right\|_Y,
           \quad x \in \R, $$
        where
        $$L_\alpha(x,y;t)
            = \frac{t}{\sqrt{4 \pi}} \int_0^\infty \frac{e^{-t^2/4u}}{u^{3/2}} (e^{-\alpha u}-1) W_u^\mathcal{L}(x,y) du,
            \quad x,y \in \R \text{ and } t>0.$$
     From \eqref{2.1} and by taking into account that  $|e^{-(\alpha +n)u}-e^{-nu}|\leq Cue^{-cu}$, $u\in (0,\infty )$, we obtain that
      \begin{align*}
            |L_\alpha(x,y;t)|
                \leq & C t e^{-c|x-y|^2}\int_0^\infty \frac{e^{-c(t^2+|x-y|^2)/u}}{u^{3/2}} \frac{|e^{-(\alpha +n)u}-e^{-nu}|}{(1-e^{-4u})^{n/2}}du \\
                \leq & C t e^{-c|x-y|^2}\int_0^\infty \frac{e^{-c (|x-y|^2+t^2)/u}e^{-cu}}{u^{1/2}(1-e^{-4u})^{n/2}} du\\
                \leq & C te^{-c|x-y|^2} \int_0^\infty \frac{e^{-c (|x-y|^2+t^2)/u}}{u^{n/2+5/4}}du\leq  C  \frac{e^{-c|x-y|^2}}{|x-y|^{n-1/2}},
                \quad x,y \in \R, \ x \neq y \text{ and } t>0.
        \end{align*}
        Hence, for every $g \in L^1(\R,Y)$,
        $$\int_\R |L_\alpha(g)(x)| dx
            \leq C  \int_\R \int_\R  \frac{e^{-c|x-y|^2}}{|x-y|^{n-1/2}} \|g(y)\|_Y dy dx
            \leq C \|g\|_{L^1(\R,Y)}. $$
        This shows that $L_\alpha$ is a bounded (sublinear) operator from $L^1(\R,Y)$ into $L^1(\R)$.\\

        The proof of this property can be finished by using Proposition~\ref{Prop2.3}.
    \end{proof}

    As usual by $H^1(\R,\B)$ we denote the classical $\B$-valued Hardy space.

    \begin{Prop}\label{Prop2.5}
        Let $Y$ be a UMD Banach space and $\alpha >-n$. The (sublinear) operator $T_\alpha^\mathcal{L}$ defined by
        $$T_\alpha^\mathcal{L}(f)(x)
            = \sup_{s >0} \left\| P_s^{\mathcal{L}+\alpha} \mathcal{G}_{\mathcal{L}+\alpha,Y}(f)(x,\cdot) \right\|_{\gamma(H,Y)},$$
        is bounded from $H^1(\R,Y)$ into $L^1(\R)$ and from $L^1(\R,Y)$ into $L^{1,\infty}(\R)$.
    \end{Prop}

    \begin{proof}
        In order to show this property we use Banach valued Calder\'on-Zygmund theory (\cite{RRT}).\\

        As in \eqref{2.0} we can see that
        $$P_t^{\mathcal{L}+\alpha}(x,y)
                \leq C \frac{t}{\left( t + |x-y| \right)^{n+1}},
            \quad x,y \in \R \text{ and } t>0.$$
        Hence, it follows that
        $$\sup_{t>0} \| P_t^{\mathcal{L}+\alpha}(g) \|_{Y}
            \leq C \sup_{t>0}  P_t(\|g\|_{Y}),
            \quad g \in L^p(\R,Y), \ 1 \leq p < \infty,$$
        and from well-known results we deduce that the maximal operator
        $$P_*^{\mathcal{L}+\alpha}(g)
            = \sup_{t>0} \| P_t^{\mathcal{L}+\alpha}(g) \|_{Y},$$
        is bounded from $L^p(\R,Y)$ into $L^p(\R)$, for every $1<p<\infty$, and from $L^1(\R,Y)$ into $L^{1,\infty}(\R)$.\\

        Moreover, according to \cite[Theorem 1]{BCCFR}
        the operator $\mathcal{G}_{\mathcal{L}+\alpha,Y}$ is bounded from
        \begin{itemize}
            \item $L^p(\R,Y)$ into $L^p(\R,\gamma(H,Y))$, $1<p<\infty$,
            \item $L^1(\R,Y)$ into $L^{1,\infty}(\R,\gamma(H,Y))$, and
            \item $H^1(\R,Y)$ into $L^1(\R,\gamma(H,Y))$.
        \end{itemize}
        Hence, if we define the operator $\T_\alpha^\mathcal{L}$ by
        $$\T_\alpha^\mathcal{L}(f)(x,s,t)
            = P_s^{\mathcal{L}+\alpha} \mathcal{G}_{\mathcal{L}+\alpha,Y}(f)(x,t), \quad x \in \R, \ s,t >0,$$
        it is bounded from $L^p(\R,Y)$ into $L^p(\R,L^\infty((0,\infty),\gamma(H,Y)))$, $1<p<\infty$, and from
        $H^1(\R,Y)$ into $L^{1,\infty}(\R,L^\infty((0,\infty),\gamma(H,Y)))$.\\

        We are going to show that $\T^\mathcal{L}_\alpha$ is bounded from $L^1(\R,Y)$ into
        $L^{1,\infty}(\R,L^\infty((0,\infty),\gamma(H,Y)))$ and from $H^1(\R,Y)$ into $L^1(\R,L^\infty((0,\infty),\gamma(H,Y)))$.\\

       We consider the function
        \begin{equation}\label{omega}
            \Omega_\alpha(x,y;s,t)
                = t \partial_t P_{t+s}^{\mathcal{L}+\alpha }(x,y),\quad x,y \in \R \text{ and } s,t>0.
        \end{equation}
        It follows from \eqref{2.0} that
        \begin{equation}\label{2.16}
            |\Omega_\alpha(x,y;s,t)|
                \leq C \frac{t}{\left( s+t + |x-y| \right)^{n+1}},
                \quad x,y \in \R \text{ and } s,t>0.
        \end{equation}
        Let $j=1, \dots, n$. By (\ref{derivW}) we get
       \begin{align}\label{2.17}
            |\partial_{x_j} \Omega_\alpha(x,y;s,t)|
                \leq & C t \int_0^\infty \frac{1}{u^{(n+4)/2}} e^{-c(|x-y|^2+(s+t)^2)/u} du \nonumber\\
                \leq & C \frac{t}{\left( s+t + |x-y| \right)^{n+2}},
                \quad x,y \in \R \text{ and } s,t>0.
        \end{align}
        By taking into account the symmetries we also have that
        \begin{align}\label{2.18}
            |\partial_{y_j} \Omega_\alpha(x,y;s,t)|
                \leq & C \frac{t}{\left( s+t + |x-y| \right)^{n+2}},
                \quad x,y \in \R \text{ and } s,t>0.
        \end{align}

        Let $N \in \mathbb{N}$ and $C([1/N,N],Y)$ be the space of continuous functions over the interval $[1/N,N]$ which take values in the Banach space $Y$.
        The function $\Omega_\alpha(x,y;s,t)$ satisfies the following Calder\'on-Zygmund type estimates
        \begin{equation}\label{2.19}
            \|\Omega_\alpha(x,y;\cdot,\cdot)\|_{C([1/N,N],H)}
                \leq \|\Omega_\alpha(x,y;\cdot,\cdot)\|_{L^\infty((0,\infty),H)}
                \leq \frac{C}{|x-y|^n},
                \quad x,y \in \R, \ x \neq y,
        \end{equation}
        and
        \begin{align}\label{2.20}
            & \|\nabla_x \Omega_\alpha(x,y;\cdot,\cdot)\|_{C([1/N,N],H)} + \|\nabla_y \Omega_\alpha(x,y;\cdot,\cdot)\|_{C([1/N,N],H)} \nonumber \\
            & \qquad \qquad \leq \|\nabla_x \Omega_\alpha(x,y;\cdot,\cdot)\|_{L^\infty((0,\infty),H)} + \|\nabla_y \Omega_\alpha(x,y;\cdot,\cdot)\|_{L^\infty((0,\infty),H)} \nonumber \\
            & \qquad \qquad \leq \frac{C}{|x-y|^{n+1}},
                \quad x,y \in \R, \ x \neq y.
        \end{align}
        Note that the constant $C$ does not depend on $N$. Indeed, by \eqref{2.16} we get
        \begin{align*}
            \|\Omega_\alpha(x,y;\cdot,\cdot)\|_{L^\infty((0,\infty),H)}
                \leq & C \sup_{s>0} \left( \int_0^\infty \frac{t}{\left( (s+t)^2 + |x-y|^2 \right)^{n+1}} dt \right)^{1/2} \\
                \leq & C \left( \int_0^\infty \frac{dt}{\left( t + |x-y| \right)^{2n+1}} \right)^{1/2}
                \leq \frac{C}{|x-y|^n},
                \quad x,y \in \R, \ x \neq y.
        \end{align*}
        and \eqref{2.19} is established. In a similar way we can deduce \eqref{2.20} from \eqref{2.17} and \eqref{2.18}.\\

        Suppose now that $g \in L^\infty_c(\R)$. By \eqref{2.19} it is clear that
        $$\int_\R \|\Omega_\alpha(x,y;\cdot,\cdot)\|_{C([1/N,N],H)} |g(y)| dy
            < \infty,
            \quad x \notin \supp(g).$$
        We define
        $$S_\alpha(g)(x)
            = \int_\R \Omega_\alpha(x,y;\cdot,\cdot) g(y) dy,
            \quad x \notin \supp(g), $$
        where the integral is understood in the $C([1/N,N],H)$-Bochner sense. We have that
        $$\left[S_\alpha(g)(x)\right](s,\cdot)
            = \int_\R \Omega_\alpha(x,y;s,\cdot) g(y) dy,
            \quad x \notin \supp(g) \text{ and } s \in [1/N,N]. $$
        Here the equality and the integral are understood in $H$ and in the $H$-Bochner sense, respectively.\\

        For every $h \in H$, we can write
        \begin{align*}
            \left\langle h , \int_\R \Omega_\alpha(x,y;s,\cdot) g(y) dy \right\rangle_{H,H}
                = & \int_\R \int_0^\infty \Omega_\alpha(x,y;s,t) h(t) \frac{dt}{t} g(y) dy \\
                = & \int_0^\infty \int_\R  \Omega_\alpha(x,y;s,t) g(y) dy h(t) \frac{dt}{t},
                \quad x \notin \supp(g) \text{ and } s \in [1/N,N]. \\
        \end{align*}
        Hence, for every $x \notin \supp(g)$ and $s>0$,
        $$\int_\R  \Omega_\alpha(x,y;s,t) g(y) dy
            = \left(\int_\R  \Omega_\alpha(x,y;s,\cdot) g(y) dy \right)(t),$$
        as elements of $H$.\\

        We have proved that
        $$P_s^{\mathcal{L}+\alpha} \mathcal{G}_{\mathcal{L}+\alpha,\mathbb{C}}(g)(x,\cdot)
            = \left[S_\alpha(g)(x)\right](s,\cdot),
            \quad x \notin \supp(g)\mbox{ and }s\in [1/N,N], $$
        in the sense of equality in $H$.\\

        Assume that $g=\sum_{j=1}^m b_j g_j$, where $b_j \in Y$ and $g_j \in L_c^\infty(\R)$, $j=1, \dots, m \in \mathbb{N}$.
        Then,
        \begin{align*}
            P_s^{\mathcal{L}+\alpha} \mathcal{G}_{\mathcal{L}+\alpha,Y}(g)(x,\cdot)
                = & \sum_{j=1}^m b_j P_s^{\mathcal{L}+\alpha} \mathcal{G}_{\mathcal{L}+\alpha,\mathbb{C}}(g_j)(x,\cdot)
                =  \sum_{j=1}^m b_j \left[S_\alpha(g_j)(x)\right](s,\cdot) \\
                = & \left(\int_\R  \Omega_\alpha(x,y;\cdot,\cdot) g(y) dy \right)(s,\cdot),
                \quad x \notin \supp(g)\mbox{ and }s\in [1/N,N],
        \end{align*}
        where the last integral is understood in the $C\left([1/N,N],\gamma(H,Y)\right)$-Bochner sense.\\

        According to Banach valued Calder\'on-Zygmund theory (see \cite{RRT}) we deduce that the operator $\T_\alpha^\LL$
        can be extended from
        \begin{itemize}
            \item $L^2(\R,Y) \cap L^1(\R,Y)$ to $L^1(\R,Y)$ as a bounded operator from  $L^1(\R,Y)$ into \\
            $L^{1,\infty}(\R,C([1/N,N],\gamma(H,Y)))$, and as
            \item  a bounded operator from  $H^1(\R,Y)$ into $L^{1}(\R,C([1/N,N],\gamma(H,Y)))$.
        \end{itemize}
        Moreover, if we denote by $\widetilde{\T}_{\alpha,N}^\LL$ the extension of $\T_\alpha^\LL$
        to $L^1(\R,Y)$ there exists $C>0$ independent of $N$ such that
        $$\|\widetilde{\T}_{\alpha,N}^\LL\|_{L^1(\R,Y) \to L^{1,\infty}(\R,C([1/N,N],\gamma(H,Y)))} \leq C$$
        and
        $$\|\widetilde{\T}_{\alpha,N}^\LL\|_{H^1(\R,Y) \to L^{1}(\R,C([1/N,N],\gamma(H,Y)))} \leq C.$$

        Let $g \in L^1(\R,Y)$ and let $(g_k)_{k \in \mathbb{N}}$ be a sequence in $L^1(\R,Y)\cap L^2(\R,Y)$ such that
        $$g_k \longrightarrow g, \quad  \text{as } k \to \infty,  \text{ in } L^1(\R,Y).$$
        It is not difficult to see that
        $$ \T_\alpha^\LL(g)(x,s,t)
            =  \mathcal{G}_{\mathcal{L}+\alpha,Y} \left(P_s^{\mathcal{L}+\alpha} (g) \right)(x,t),
            \quad x \in \R \text{ and } s,t>0,$$
        and
        $$ \T_\alpha^\LL(g_k)(x,s,t)
            =  \mathcal{G}_{\mathcal{L}+\alpha,Y} \left(P_s^{\mathcal{L}+\alpha} (g_k) \right)(x,t),
            \quad x \in \R, \ s,t>0 \text{ and } k \in \mathbb{N}.$$
        Hence, since $P_s^{\mathcal{L}+\alpha}$ is bounded from $L^1(\R,Y)$ into itself,
        for every $s>0$, and $\mathcal{G}_{\mathcal{L}+\alpha,Y}$
        is bounded from $L^1(\R,Y)$ into $L^{1,\infty}(\R,\gamma(H,Y))$ (\cite[Theorem 1]{BCCFR}),
        $$ \T_\alpha^\LL(g_k)(\cdot,s,\cdot)
            \longrightarrow \T_\alpha^\LL(g)(\cdot,s,\cdot),
            \quad \text{ as } k \to \infty, \text{ in } L^{1,\infty}(\R,\gamma(H,Y)),$$
        for every $s>0$. Moreover, we can find a subsequence $(g_{k_{\ell}})_{\ell \in \mathbb{N}}$
        of $(g_{k})_{k \in \mathbb{N}}$ verifying that for every $s \in \mathbb{Q}$,
        $$\T_\alpha^\LL(g_{k_\ell})(x,s,\cdot)
            \longrightarrow \T_{\alpha}^\LL(g)(x,s,\cdot),
            \quad \text{ as } \ell \to \infty, \text{ in } \gamma(H,Y),$$
        a.e. $x \in \R$. On the other hand,
        $$\T_\alpha^\LL(g_{k_\ell}) = \widetilde{\T}_{\alpha,N}^\LL(g_{k_\ell})
            \longrightarrow \widetilde{\T}_{\alpha,N}^\LL(g),
            \quad \text{ as } \ell \to \infty, \text{ in } L^{1,\infty}(\R, C([1/N,N],\gamma(H,Y))),$$
        and then, there exists a subsequence $(g_{k_{\ell_j}})_{j \in \mathbb{N}}$
        of $(g_{k_{\ell}})_{\ell \in \mathbb{N}}$ such that, for every $s \in [1/N,N]$,
        $$\T_\alpha^\LL(g_{k_{\ell_j}})(x,s,\cdot)
            \longrightarrow \widetilde{\T}_{\alpha,N}^\LL(g)(x,s,\cdot),
            \quad \text{ as } j \to \infty, \text{ in } \gamma(H,Y),$$
        a.e. $x \in \R$. Thus, for every $s \in [1/N,N] \cap \mathbb{Q}$,
        $$\widetilde{\T}_{\alpha,N}^\LL(g)(x,s,\cdot)
            = \T_{\alpha}^\LL(g)(x,s,\cdot),
            \quad \text{a.e. } x \in \R, \text{ in }\gamma(H,Y).$$

        Finally,
        \begin{align*}
           & \left| \left\{ x \in \R : \sup_{s>0} \|\T_{\alpha}^\LL(g)(x,s,\cdot)\|_{\gamma(H,Y)} > \lambda \right\} \right|
                \leq  \left| \bigcup_{N \in \mathbb{N}} \left\{ x \in \R : \sup_{s \in [1/N,N]} \|\T_{\alpha}^\LL(g)(x,s,\cdot)\|_{\gamma(H,Y)} > \lambda \right\} \right| \\
           & \qquad \qquad = \lim_{N \to \infty} \left| \left\{ x \in \R : \sup_{s \in [1/N,N] \cap \mathbb{Q}} \|\T_{\alpha}^\LL(g)(x,s,\cdot)\|_{\gamma(H,Y)} > \lambda \right\} \right| \\
           & \qquad \qquad = \lim_{N \to \infty} \left| \left\{ x \in \R : \sup_{s \in [1/N,N] \cap \mathbb{Q}} \|\widetilde{\T}_{\alpha ,N}^\LL(g)(x,s,\cdot)\|_{\gamma(H,Y)} > \lambda \right\} \right| \\
           & \qquad \qquad \leq \frac{C}{\lambda} \|g\|_{L^1(\R,Y)} , \quad \lambda >0,
        \end{align*}
        and we conclude that $\T_\alpha^\LL$ is bounded from $L^1(\R,Y)$ into $L^{1,\infty}(\R,L^\infty((0,\infty),\gamma(H,Y)))$.\\

        By proceeding in a similar way we can show that $\T_\alpha^\LL$ is also bounded from $H^1(\R,Y)$ into
        $L^{1}(\R,L^\infty((0,\infty),\gamma(H,Y)))$.
    \end{proof}

    We now establish that $\mathcal{G}_{\mathcal{L}+\alpha,\B}$ is bounded from $H^1_\LL(\R,\B)$ into  $H^1_\LL(\R,\gamma(H,\B))$.
    According to Proposition~\ref{Prop2.4} it is sufficient to show that
    $\mathcal{G}_{\mathcal{L}+\alpha,\B}(f) \in L^1(\R,\gamma(H,\B))$, for every $f \in H^1_\LL(\R,\B)$, and that the operator
    $$T_\alpha^\mathcal{L}(f)(x)
            = \sup_{s >0} \left\| P_s^{\mathcal{L}+\alpha} \mathcal{G}_{\mathcal{L}+\alpha,\B}(f)(x,\cdot) \right\|_{\gamma(H,\B)},$$
    is bounded from $H^1_\LL(\R,\B)$ into $L^1(\R)$.\\

    First of all, we are going to see that $\mathcal{G}_{\mathcal{L}+\alpha,\B}$ is a bounded operator from
    $H^1_\LL(\R,\B)$ into $L^1(\R,$ $\gamma(H,\B))$. By \cite[Theorem 1]{BCCFR},
    $\mathcal{G}_{\mathcal{L}+\alpha,\B}$
    is bounded from $H^1(\R,\B)$ into $L^1(\R,\gamma(H,\B))$. Hence, if $a$ is an atom for $H^1_\LL(\R,\B)$
    such that $\int_\R a(x)dx=0$, then
    $$\|\mathcal{G}_{\mathcal{L}+\alpha,\B}(a)\|_{L^1(\R,\gamma(H,\B))}
        \leq C \|a\|_{L^1(\R,\B)}
        \leq C,$$
    where $C>0$ does not depend on the atom $a$.\\

    Suppose now that $a$ is an atom for $H^1_\LL(\R,\B)$ such that $\supp(a) \subset B=B(x_0,r_0)$, where $x_0 \in  \R$ and
    $\rho(x_0)/2 \leq r_0 \leq \rho(x_0)$, and that $\|a\|_{L^\infty(\R,\B)} \leq |B|^{-1} $.
    Since $\mathcal{G}_{\mathcal{L}+\alpha,\B}$ is a bounded operator from $L^2(\R,\B)$ into $L^2(\R,\gamma(H,\B))$
    (\cite[Theorem 1]{BCCFR}),
    we have that
    \begin{align}\label{2.21}
        \int_{B^*} \|\mathcal{G}_{\mathcal{L}+\alpha,\B}(a)(x,\cdot)\|_{\gamma(H,\B)} dx
            \leq & |B^*|^{1/2} \left( \int_{B^*} \|\mathcal{G}_{\mathcal{L}+\alpha,\B}(a)(x,\cdot)\|_{\gamma(H,\B)}^2 dx\right)^{1/2} \nonumber \\
            \leq & C |B|^{1/2} \left( \int_{B} \|a(x)\|_\B^2 dx\right)^{1/2}
            \leq C,
    \end{align}
    being $B^*=B(x_0,2r_0)$.\\

    Moreover, if $y \in B$ and $x \notin B^*$, it follows that
    $|x-y|\geq r_0 \geq \rho(x_0)/2$ and $\rho(y) \sim \rho (x_0)$.
    Then, by taking into account \eqref{2.0.1}, \eqref{f1} and \eqref{A}  we get
    \begin{align}\label{2.22}
        \int_{\R \setminus B^*} \|\mathcal{G}_{\mathcal{L}+\alpha,\B}(a)(x,\cdot)\|_{\gamma(H,\B)} dx
            \leq & \int_{\R \setminus B^*} \int_B \|t \partial_t P_t^{\LL + \alpha}(x,y)\|_H \|a(y)\|_\B dy dx \nonumber \\
            \leq & C \int_{\R \setminus B^*} \int_B \frac{e^{-c(|x-y|^2+|y||x-y|)}}{|x-y|^n} \|a(y)\|_\B dy dx \nonumber \\
            \leq & C \int_B \|a(y)\|_\B  \sum_{j=0}^\infty \int_{2^{j-1}\rho(x_0) \leq |x-y| < 2^{j}\rho(x_0)}
                \frac{dx dy}{|x-y|^{n+1/2}\rho(x_0)^{-1/2}}   \nonumber \\
            \leq & C   \sum_{j=0}^\infty \frac{1}{(2^{j}\rho(x_0))^{1/2}\rho(x_0)^{-1/2}}
            \leq C \sum_{j=0}^\infty \frac{1}{2^{j/2}}
            \leq C.
    \end{align}
    From \eqref{2.21} and \eqref{2.22} we infer that
    $$\|\mathcal{G}_{\mathcal{L}+\alpha,\B}(a)\|_{L^1(\R,\gamma(H,\B))}
        \leq C,$$
    where $C>0$ does not depend on $a$.\\

    We consider $f=\sum_{j=1}^\infty \lambda_j a_j$, where $a_j$ is an atom for $H^1_\LL(\R,\B)$ and $\lambda_j \in  \mathbb{C}$,
    $j \in \mathbb{N}$, being $\sum_{j=1}^\infty |\lambda_j|<\infty$. The series converges in $L^1(\R,\B)$.
    Hence, as a consequence of \cite[Theorem 1]{BCCFR}, we have that
    $$\mathcal{G}_{\mathcal{L}+\alpha,\B}(f)
        =  \sum_{j=1}^\infty \lambda_j \mathcal{G}_{\mathcal{L}+\alpha,\B}(a_j),$$
    as elements of $L^{1,\infty}(\R,\gamma(H,\B))$. Also,
    \begin{align*}
        \left\| \mathcal{G}_{\mathcal{L}+\alpha,\B}(f) \right\|_{L^1(\R,\gamma(H,\B))}
        \leq & \sum_{j=1}^\infty |\lambda_j| \left\| \mathcal{G}_{\mathcal{L}+\alpha,\B}(a_j) \right\|_{L^1(\R,\gamma(H,\B))}
        \leq  C \sum_{j=1}^\infty |\lambda_j|,
    \end{align*}
    where $C>0$ does not depend on $f$. Thus,
    $$\| \mathcal{G}_{\mathcal{L}+\alpha,\B}(f)\|_{L^1(\R,\gamma(H,\B))}
        \leq C \|f\|_{H^1_\LL(\R,\B)}.$$

    Finally, to show that $T_\alpha^\mathcal{L}$ is bounded from $H^1_\LL(\R,\B)$ into $L^1(\R)$ we can proceed as above
    by considering the action of the operator on the two types of atoms of $H^1_\LL(\R,\B)$, and taking in mind the following facts,
    which can be deduced from the proof of Proposition~\ref{Prop2.5}:
    \begin{itemize}
        \item $T_\alpha^\mathcal{L}$ is bounded from $H^1(\R,\B)$ into $L^1(\R)$,
        \item $T_\alpha^\mathcal{L}$ is bounded from $L^2(\R,\B)$ into $L^2(\R)$,
        \item $P_s^{\mathcal{L}+\alpha} \mathcal{G}_{\mathcal{L}+\alpha,\B}$ can be associated to an integral operator with kernel $\Omega_\alpha$ (see \eqref{omega}) verifying that
        $$\sup_{s>0} \| \Omega_\alpha(x,y,s,\cdot)\|_H \leq C \frac{e^{-c(|x-y|^2+|y| |x-y|)}}{|x-y|^n}, \quad x, y \in \R, \ x \neq y,$$
        \item $T_\alpha^\mathcal{L}$ is bounded from $L^1(\R,\B)$ into $L^{1,\infty}(\R)$.
    \end{itemize}

%\newpage
%%%%%%%%%%%%%%%%%
    \subsection{}\label{subsec:2.3}

    Our next objective is to see that there exists $C>0$ such that
    \begin{equation}\label{2.B.0}
        \|f\|_{BMO_\mathcal{L}(\R,\B)}
            \leq C \|\mathcal{G}_{\mathcal{L}+\alpha,\B}(f)\|_{BMO_\mathcal{L}(\R,\gamma(H,\B))},
            \quad f \in BMO_\mathcal{L}(\R,\B).
    \end{equation}
    In order to prove this we need to establish the following polarization equality.

    \begin{Prop}\label{Prop2.1}
        Let $\B$ be a Banach space. If $a \in L^\infty_c(\R) \otimes \B^*$ and
        $f \in BMO_\mathcal{L}(\R,\B)$,  then
        $$\int_0^\infty \int_{\R}
                \langle \mathcal{G}_{\mathcal{L}+\alpha,\B^*}(a)(x,t) , \mathcal{G}_{\mathcal{L}+\alpha,\B}(f)(x,t) \rangle_{\B^*,\B} \frac{dx dt}{t}
            = \frac{1}{4} \int_{\R} \langle a(x) , f(x) \rangle_{\B^*,\B}dx.$$
    \end{Prop}

    \begin{proof}
        Firstly we consider $a \in L^\infty_c(\R)$ and $f \in BMO_\mathcal{L}(\R)$.
        In order to prove that
        \begin{equation}\label{2.B}
            \int_0^\infty \int_{\R}
                    \mathcal{G}_{\mathcal{L}+\alpha,\mathbb{C}}(a)(x,t) \mathcal{G}_{\mathcal{L}+\alpha,\mathbb{C}}(f)(x,t) \frac{dx dt}{t}
                = \frac{1}{4} \int_{\R} a(x) f(x) dx,
        \end{equation}
        we use the ideas developed in the proof of \cite[Lemma 4]{DGMTZ}.\\

        According to \cite[Lemma 5]{DGMTZ} we can write
        \begin{equation}\label{2.1.3}
            \int_0^\infty \int_{\R}
                    \left| \mathcal{G}_{\mathcal{L}+\alpha,\mathbb{C}}(a)(x,t)\right| \left|\mathcal{G}_{\mathcal{L}+\alpha,\mathbb{C}}(f)(x,t)\right| \frac{dx dt}{t}
                \leq C \|S_\alpha(a)\|_{L^1(\R)} \|I_\alpha(f)\|_{L^\infty(\R)} ,
        \end{equation}
        where
        $$S_\alpha(a)(x)
            = \left( \int_0^\infty \int_{|x-y|<t} \left| \mathcal{G}_{\mathcal{L}+\alpha,\mathbb{C}}(a)(y,t)\right|^2 \frac{dydt}{t^{n+1}} \right)^{1/2},
            \quad x \in \R, $$
        and
        $$I_\alpha(f)(x)
            = \sup_{B \ni x}\left( \frac{1}{|B|} \int_0^{r(B)} \int_B \left| \mathcal{G}_{\mathcal{L}+\alpha,\mathbb{C}}(f)(y,t)\right|^2 \frac{dydt}{t} \right)^{1/2},
            \quad x \in \R.$$
        Here $B$ represents a ball in $\R$ and $r(B)$ is its radius.\\

        We are going to show that the area integral operator $S_\alpha$ is bounded from
        $H^1_\mathcal{L}(\R)$ into $L^1(\R)$. According to \cite[Theorem 8.2]{HLMMY}
        $S_0$ is bounded from $H^1_\mathcal{L}(\R)$ into $L^1(\R)$. Then, it is
        sufficient to see that $S_\alpha - S_0$ is bounded from $L^1(\R)$ into itself.\\

        By using Minkowski's inequality we obtain
        \begin{align*}
             &\left( \int_0^\infty \int_{|x-y|<t}
                    \left| \mathcal{G}_{\mathcal{L}+\alpha,\mathbb{C}}(g)(y,t) -  \mathcal{G}_{\mathcal{L},\mathbb{C}}(g)(y,t)\right|^2 \frac{dydt}{t^{n+1}} \right)^{1/2} \\
             & \qquad \qquad \leq \int_{\R} |g(z)|  \left( \int_0^\infty \int_{|x-y|<t}
                    \left| t \partial_t \left[P_t^{\mathcal{L}+\alpha}(y,z) - P_t^{\mathcal{L}}(y,z)\right]\right|^2 \frac{dydt}{t^{n+1}} \right)^{1/2}dz,
                \quad g \in L^1(\R).
        \end{align*}
        Since,
        $$t \partial_t \left[P_t^{\mathcal{L}+\alpha}(y,z) - P_t^{\mathcal{L}}(y,z)\right]
            = \frac{t}{\sqrt{4 \pi}} \int_0^\infty  \frac{e^{-t^2/4s}}{s^{3/2}}
                \left( 1-\frac{t^2}{2s}\right) (e^{-\alpha s}-1) W_s^\mathcal{L}(y,z) ds,
            \quad y,z \in \R, \ t>0,$$
        by employing Minkowski's inequality and (\ref{2.1}) it follows that
        \begin{align*}
            & \left( \int_0^\infty \int_{|x-y|<t}
                    \left| t \partial_t \left[P_t^{\mathcal{L}+\alpha}(y,z) - P_t^{\mathcal{L}}(y,z)\right]\right|^2 \frac{dydt}{t^{n+1}} \right)^{1/2}\\
            & \qquad \qquad \leq C \int_0^\infty  \frac{|e^{-\alpha s}-1|}{s^{3/2}}
                    \left( \int_0^\infty \int_{|x-y|<t}
                    \left| t e^{-t^2/8s} W_s^\mathcal{L}(y,z) \right|^2 \frac{dydt}{t^{n+1}} \right)^{1/2} ds\\
    & \qquad \qquad \leq C \int_0^\infty  \frac{|e^{-(\alpha +n)s}-e^{-ns}|}{s^{3/2}(1-e^{-4s})^{n/2}}
                    \left( \int_0^\infty \int_{|x-y|<t}
                     e^{-c(t^2+|y-z|^2)/s} \frac{dydt}{t^{n-1}} \right)^{1/2} ds,
            \quad x,z \in \R.
        \end{align*}By taking into account again that $|e^{-(\alpha +n)s}-e^{-ns}|\leq Cse^{-cs}$, $s\in (0,\infty )$, and that $t^2+|z-y|^2\geq (t^2+|z-x|^2)/4$, when $|x-y|<t$, we can write
     \begin{align*}
            & \left( \int_0^\infty \int_{|x-y|<t}
                    \left| t \partial_t \left[P_t^{\mathcal{L}+\alpha}(y,z) - P_t^{\mathcal{L}}(y,z)\right]\right|^2 \frac{dydt}{t^{n+1}} \right)^{1/2}\\
            & \qquad \qquad \leq C \int_0^\infty  \frac{e^{-cs}e^{-c|x-z|^2/s}}{s(1-e^{-4s})^{n/2}}
                    \left( \int_0^\infty \int_{|x-y|<t}
                     e^{-ct^2/s} \frac{dydt}{t^{n-1}} \right)^{1/2} ds\\
    & \qquad \qquad \leq C \int_0^\infty  \frac{e^{-cs}e^{-c|x-z|^2/s}}{s^{(n+1)/2}}ds,\quad x,z\in \mathbb{R}^n.
         \end{align*}
       Then,
        $$\int_{\R}  \left( \int_0^\infty \int_{|x-y|<t}
                \left| t \partial_t \left[P_t^{\mathcal{L}+\alpha}(y,z) - P_t^{\mathcal{L}}(y,z)\right]\right|^2 \frac{dydt}{t^{n+1}} \right)^{1/2}dx
            \leq C \int_0^\infty  \frac{e^{-cs}}{s^{1/2}}ds\leq C, \quad z \in \R.$$
        Hence, the operator $S_\alpha-S_0$ is bounded from $L^1(\R)$ into itself.\\

        Our next objective is to see that $I_\alpha(f) \in L^\infty(\R)$. Let $x_0 \in \R$ and
        $r_0>0$. We denote by $B$ the ball $B(x_0,r_0)$ and we decompose $f$ as follows
        $$f
            = (f-f_{B^*}) \chi_{B^*} + (f-f_{B^*}) \chi_{\R \setminus B^*} + f_{B^*}
            = f_1 + f_2 + f_3, $$
        where $B^*=B(x_0,2r_0)$.\\

        According to \cite[(4)]{BCCFR}, since $\gamma(H,\mathbb{C})=H$, we can write
        \begin{align}\label{2.4}
            \frac{1}{|B|} \int_0^{r_0} \int_B  \left| \mathcal{G}_{\mathcal{L}+\alpha,\mathbb{C}}(f_1)(y,t)\right|^2 \frac{dydt}{t}
                \leq & \frac{1}{|B|} \int_{\R} \int_0^\infty  \left| \mathcal{G}_{\mathcal{L}+\alpha,\mathbb{C}}(f_1)(y,t)\right|^2 \frac{dtdy}{t} \nonumber \\
                \leq & \frac{C}{|B|} \int_{B^*} \left| f(x) - f_{B^*}\right|^2 dx
                \leq C \|f\|^2_{BMO_\mathcal{L}(\R)}.
        \end{align}
        By using \eqref{2.0}  we can proceed as in \cite[p. 338]{DGMTZ} to obtain
        \begin{align}\label{2.5}
            \frac{1}{|B|} \int_0^{r_0} \int_B  \left| \mathcal{G}_{\mathcal{L}+\alpha,\mathbb{C}}(f_2)(y,t)\right|^2 \frac{dydt}{t}
                \leq C \|f\|^2_{BMO_\mathcal{L}(\R)}.
        \end{align}
        If $r_0 \geq \rho(x_0)$, since $\mathcal{G}_{\mathcal{L}+\alpha,\mathbb{C}}(1) \in L^\infty(\R,H)$
        (see Subsection~\ref{subsec:2.1}), then
        \begin{align}\label{f3}
            \frac{1}{|B|} \int_0^{r_0} \int_B  \left| \mathcal{G}_{\mathcal{L}+\alpha,\mathbb{C}}(f_3)(y,t)\right|^2 \frac{dydt}{t}
                \leq &  \frac{|f_{B^*}|^2}{|B|} \int_B  \left\| \mathcal{G}_{\mathcal{L}+\alpha,\mathbb{C}}(1)(y,\cdot)\right\|_{H}^2 dy
                \leq  C |f_{B^*}|^2
                \leq C \|f\|^2_{BMO_\mathcal{L}(\R)}.
        \end{align}

        Suppose now that $r_0 < \rho(x_0)$. According to \eqref{2.1.2}, we have that
        $$\mathcal{G}_{\mathcal{L}+\alpha,\mathbb{C}}(1)(x,t)
            = \frac{t^2}{\sqrt{4\pi}} \int_0^\infty \frac{e^{-u}}{u^{3/2}}
                \partial_z W_z^{\LL + \alpha}(1)(x)_{|{ z=t^2/4u}} du,
            \quad x \in \R \text{ and } t>0.$$
     By \eqref{2.1.1} it follows that, for every $x \in \R \text{ and } z>0$,
        $$|\partial_z W_z^{\mathcal{L}+\alpha }(1)(x)|
            \leq C \frac{e^{-(\alpha +n)z}e^{-c(1-e^{-4z})|x|^2}}{(\rho (x))^2}\leq  C \frac{e^{-cz}\max\{e^{-cz/(\rho (x)^2)}, e^{-c/(\rho (x))^2}\}}{(\rho (x))^2}\leq C\frac{1}{(\rho(x))^{1/2}z^{3/4}}.
 $$
        Then, we conclude that
        \begin{align*}
            |\mathcal{G}_{\mathcal{L}+\alpha,\mathbb{C}}(1)(x,t)|
                \leq & C \left(\frac{t}{\rho(x)} \right)^{1/2},
                \quad x \in \R \text{ and } t>0.
        \end{align*}
        The arguments developed in \cite[p. 339]{DGMTZ} allow us to obtain
        \begin{equation}\label{2.6}
            \frac{1}{|B|} \int_0^{r_0} \int_B  \left| \mathcal{G}_{\mathcal{L}+\alpha,\mathbb{C}}(f_3)(y,t)\right|^2 \frac{dydt}{t}
                \leq C \|f\|^2_{BMO_\mathcal{L}(\R)}.
        \end{equation}
        Putting together \eqref{2.4}, \eqref{2.5},\eqref{f3} and \eqref{2.6} we get
        $$\frac{1}{|B|} \int_0^{r_0} \int_B  \left| \mathcal{G}_{\mathcal{L}+\alpha,\mathbb{C}}(f)(y,t)\right|^2 \frac{dydt}{t}
                \leq C \|f\|^2_{BMO_\mathcal{L}(\R)},$$
        where $C$ does not depend on $B$, and we prove that $I_\alpha(f) \in L^\infty(\R)$.\\

        Since $a \in H^1_\LL(\R)$, from \eqref{2.1.3} we deduce that
        \begin{equation}\label{2.6.1}
            \int_0^\infty \int_{\R}
                    \left| \mathcal{G}_{\mathcal{L}+\alpha,\mathbb{C}}(a)(x,t)\right| \left|\mathcal{G}_{\mathcal{L}+\alpha,\mathbb{C}}(f)(x,t)\right| \frac{dx dt}{t}
                < \infty.
        \end{equation}
        Then,
        $$\int_0^\infty \int_{\R}
                    \mathcal{G}_{\mathcal{L}+\alpha,\mathbb{C}}(a)(x,t) \mathcal{G}_{\mathcal{L}+\alpha,\mathbb{C}}(f)(x,t) \frac{dx dt}{t}
                = \lim_{N \to \infty} \int_{1/N}^N \int_{\R}
                    \mathcal{G}_{\mathcal{L}+\alpha,\mathbb{C}}(a)(x,t) \mathcal{G}_{\mathcal{L}+\alpha,\mathbb{C}}(f)(x,t) \frac{dx dt}{t}.$$
        Let $N \in \mathbb{N}$. By interchanging the order of integration we obtain
        $$\int_{1/N}^N \int_{\R}
                    \mathcal{G}_{\mathcal{L}+\alpha,\mathbb{C}}(a)(x,t) \mathcal{G}_{\mathcal{L}+\alpha,\mathbb{C}}(f)(x,t) \frac{dx dt}{t}
            = \int_{\R} f(y) \int_{1/N}^N
                    \mathcal{G}_{\mathcal{L}+\alpha,\mathbb{C}} \left(\mathcal{G}_{\mathcal{L}+\alpha,\mathbb{C}}(a)(\cdot,t_1) \right)(y,t)_{|{t_1=t}} \frac{dt dy}{t}. $$
        We are going to justify this interchange in the order of integration. In order to do it we will see that
        \begin{equation}\label{2.7}
            \int_{1/N}^N \int_{\R} |f(y)| \int_{\R}
                    \left|t \partial_t P_t^{\LL + \alpha}(x,y)\mathcal{G}_{\mathcal{L}+\alpha,\mathbb{C}}(a)(x,t) \right| \frac{dx dy dt}{t}
                < \infty.
        \end{equation}
        By using \eqref{2.0} it follows that
        \begin{align*}
            & \int_{\R} \left|t \partial_t P_t^{\LL + \alpha}(x,y) \right|
                    \int_{\R} \left|t \partial_t P_t^{\LL + \alpha}(x,z) \right| |a(z)| dz dx \\
            & \qquad \qquad \leq C \int_{\R} |a(z)| \int_{\R}
                    \frac{t}{\left( |x-z|^2+t^2\right)^{(n+1)/2}} \frac{t}{\left( |x-y|^2+t^2\right)^{(n+1)/2}}  dx dz\\
            & \qquad \qquad \leq C \int_{\R} |a(z)| \frac{t}{\left(t+ |z-y| \right)^{n+1}} dz,
            \quad x,y \in \R \text{ and } t>0.
        \end{align*}
        Suppose that $\supp(a) \subset B=B(0,R)$. We have that
        \begin{align}\label{2.8}
            & \int_{\R} \left|t \partial_t P_t^{\LL + \alpha}(x,y) \right|
                    \int_{\R} \left|t \partial_t P_t^{\LL + \alpha}(x,z) \right| |a(z)| dz dx
            \leq C \|a\|_{L^\infty(\R)}
            \leq \frac{C}{(1+|y|)^{n+1}},
            \ y \in B^* \text{ and } t>0.
        \end{align}
        On the other hand, if $y \notin B^*=B(0,2R)$ and $z \in B$, then $|z-y| \geq |y|/2$. Hence, we get
        \begin{align}\label{2.9}
            & \int_{\R} \left|t \partial_t P_t^{\LL + \alpha}(x,y) \right|
                    \int_{\R} \left|t \partial_t P_t^{\LL + \alpha}(x,z) \right| |a(z)| dz dx
             \leq C R^n \|a\|_{L^\infty(\R)} \frac{t}{(t+|y|)^{n+1}},
            \ y \notin B^* \text{ and } t>0.
        \end{align}
        Since $f \in BMO_\LL(\R)$, \eqref{2.8} and \eqref{2.9} imply \eqref{2.7}.\\

        By taking into account that $a \in L^2(\R)$ we can write, for every $x \in \R$ and $t>0$,
        \begin{align*}
            \mathcal{G}_{\mathcal{L}+\alpha,\mathbb{C}} \left(\mathcal{G}_{\mathcal{L}+\alpha,\mathbb{C}}(a)(\cdot,t_1) \right)(x,t)_{|{t_1=t}}
                = & - \mathcal{G}_{\mathcal{L}+\alpha,\mathbb{C}} \left(
                    \sum_{k \in \mathbb{N}^n} t_1 \sqrt{2|k|+n+\alpha} e^{- t_1 \sqrt{2|k|+n+\alpha}}
                    \langle a , h_k \rangle h_k \right) (x,t)_{|{t_1=t}} \\
                = & t^2 \sum_{k \in \mathbb{N}^n} (2|k|+n+\alpha) e^{- 2 t \sqrt{2|k|+n+\alpha}}
                    \langle a , h_k \rangle h_k(x).
        \end{align*}
        Note that the last series converges uniformly in $(x,t) \in  \R\times  [a,b] $, for every $0<a<b<\infty$.
        We have that
        \begin{align*}
            & \int_{1/N}^N \mathcal{G}_{\mathcal{L}+\alpha,\mathbb{C}} \left(\mathcal{G}_{\mathcal{L}+\alpha,\mathbb{C}}(a)(\cdot,t_1) \right)(y,t)_{|{t_1=t}} \frac{dt}{t}
                =  \sum_{k \in \mathbb{N}^n} \langle a , h_k \rangle h_k(y)
                    (2|k|+n+\alpha) \int_{1/N}^N t e^{- 2 t \sqrt{2|k|+n+\alpha}} dt \\
            & \qquad \qquad = \sum_{k \in \mathbb{N}^n} \langle a , h_k \rangle h_k(y)
                    \Big[ -\frac{1}{2} \sqrt{2|k|+n+\alpha}
                            \left( N e^{- 2 N\sqrt{2|k|+n+\alpha}} - \frac{1}{N} e^{- \frac{2}{N} \sqrt{2|k|+n+\alpha}} \right) \\
            & \qquad \qquad \qquad  -\frac{1}{4}\left( e^{- 2 N\sqrt{2|k|+n+\alpha}} - e^{- \frac{2}{N} \sqrt{2|k|+n+\alpha}} \right)\Big]\\
            & \qquad \qquad = -\frac{1}{4} \left[ P_{2N}^{\LL + \alpha}(a)(y) -  P_{2/N}^{\LL + \alpha}(a)(y) \right]
                              -\frac{1}{4} \left[ \mathcal{G}_{\LL + \alpha,\mathbb{C}}(a)(y,2N) - \mathcal{G}_{\LL + \alpha,\mathbb{C}}(a)(y,2/N)\right],
            \quad y \in \R.
        \end{align*}
        According to \eqref{2.0} it follows that
        \begin{align*}
            \sup_{t>0} \left| \mathcal{G}_{\LL + \alpha,\mathbb{C}}(a)(y,t) \right|
                \leq & C  \sup_{t>0} \int_{\R}  \frac{t|a(z)|}{\left(t+ |z-y| \right)^{n+1}} dz
                \leq  C  \|a\|_{L^\infty(\R)}
                \leq \frac{C}{(1+|y|)^{n+1}},
                \ y \in B^*,
        \end{align*}
        and by proceeding as in \eqref{2.0} and using \eqref{2.1}, we get
        \begin{align*}
            \sup_{t>0} \left| \mathcal{G}_{\LL + \alpha,\mathbb{C}}(a)(y,t) \right|
                \leq & C  \sup_{t>0} \int_{B} |a(z)| \frac{t e^{-c|y||z-y|}}{\left(t+ |z-y| \right)^{n+1}} dz \\
                \leq & C  \|a\|_{L^\infty(\R)} e^{-c|y|^2} \int_{\R} \frac{t}{\left(t+ |z-y| \right)^{n+1}} dz
                \leq \frac{C}{(1+|y|)^{n+1}},
                \quad y \notin B^*,
        \end{align*}
        In a similar way we can prove that
        $$\sup_{t>0} \left| P_{t}^{\LL + \alpha}(a)(y)\right|
            \leq \frac{C}{(1+|y|)^{n+1}},  \quad y \in \R.$$

        We conclude that
        $$\sup_{N \in \mathbb{N}} \left| \int_{1/N}^N \mathcal{G}_{\mathcal{L}+\alpha,\mathbb{C}} \left(\mathcal{G}_{\mathcal{L}+\alpha,\mathbb{C}}(a)(\cdot,t_1) \right)(y,t)_{|{t_1=t}} \frac{dt}{t} \right|
            \leq \frac{C}{(1+|y|)^{n+1}},  \quad y \in \R.$$
        Hence, for every increasing sequence $(N_m)_{m \in \mathbb{N}} \subset \mathbb{N}$, we have that
        $$\int_0^\infty \int_{\R}
                    \mathcal{G}_{\mathcal{L}+\alpha,\mathbb{C}}(a)(x,t) \mathcal{G}_{\mathcal{L}+\alpha,\mathbb{C}}(f)(x,t) \frac{dx dt}{t}
            = \int_{\R} f(y) \lim_{m \to \infty} \int_{1/N_m}^{N_m}
                    \mathcal{G}_{\mathcal{L}+\alpha,\mathbb{C}} \left(\mathcal{G}_{\mathcal{L}+\alpha,\mathbb{C}}(a)(\cdot,t_1) \right)(y,t)_{|{t_1=t}} \frac{dt dy}{t}. $$
        because $f \in BMO_\LL(\R)$.\\

        Then, \eqref{2.B} will be proved when we show that
        \begin{equation}\label{2.10}
            \lim_{N \to \infty} \int_{1/N}^{N}
                    \mathcal{G}_{\mathcal{L}+\alpha,\mathbb{C}} \left(\mathcal{G}_{\mathcal{L}+\alpha,\mathbb{C}}(a)(\cdot,t_1) \right)(y,t)_{|{t_1=t}} \frac{dt}{t}
                = \frac{a(y)}{4} ,
                \quad \text{in } L^2(\R).
        \end{equation}
        In order to see that \eqref{2.10} holds we use Plancherel equality to get
        \begin{align*}
            & \left\| \int_{1/N}^{N} \mathcal{G}_{\mathcal{L}+\alpha,\mathbb{C}} \left(\mathcal{G}_{\mathcal{L}+\alpha,\mathbb{C}}(a)(\cdot,t_1) \right)(y,t)_{|{t_1=t}} \frac{dt}{t}
                    - \frac{a(y)}{4}  \right\|_{L^2(\R)}^2 \\
            & \qquad \qquad = \sum_{k \in \mathbb{N}^n} |\langle a , h_k \rangle |^2
                    \Big| \frac{\sqrt{2|k|+n+\alpha}}{2}
                            \left(- N e^{- 2 N\sqrt{2|k|+n+\alpha}} + \frac{1}{N} e^{- \frac{2}{N} \sqrt{2|k|+n+\alpha}} \right) \\
            & \qquad \qquad \qquad  -\frac{1}{4}\left( e^{- 2 N\sqrt{2|k|+n+\alpha}} - e^{- \frac{2}{N} \sqrt{2|k|+n+\alpha}} \right) - \frac{1}{4}\Big|^2.
        \end{align*}
        The dominated convergence Theorem leads to
        $$\lim_{N \to \infty} \left\| \int_{1/N}^{N} \mathcal{G}_{\mathcal{L}+\alpha,\mathbb{C}} \left(\mathcal{G}_{\mathcal{L}+\alpha,\mathbb{C}}(a)(\cdot,t_1) \right)(y,t)_{|{t_1=t}} \frac{dt}{t}
                    - \frac{a(y)}{4}  \right\|_{L^2(\R)}^2
            = 0.$$
        Thus, the proof of \eqref{2.B} is finished.\\

        Suppose now that $f \in BMO_\LL(\R,\B)$ and $a=\sum_{j=1}^m a_j b_j$, where $a_j \in L_c^\infty(\R)$
        and $b_j \in \B^*$, $j=1, \dots, m \in \mathbb{N}$. We have that
        \begin{align*}
            & \int_0^\infty \int_{\R}
                    \langle \mathcal{G}_{\mathcal{L}+\alpha,\B^*}(a)(x,t) , \mathcal{G}_{\mathcal{L}+\alpha,\B}(f)(x,t) \rangle_{\B^*,\B}\frac{dx dt}{t} \\
            & \qquad \qquad = \sum_{j=1}^m \int_0^\infty \int_{\R} \mathcal{G}_{\mathcal{L}+\alpha,\mathbb{C}}(a_j)(x,t)
                     \langle b_j , \mathcal{G}_{\mathcal{L}+\alpha,\B}\left(f\right)(x,t) \rangle_{\B^*,\B}\frac{dx dt}{t} \\
            & \qquad \qquad = \sum_{j=1}^m \int_0^\infty \int_{\R} \mathcal{G}_{\mathcal{L}+\alpha,\mathbb{C}}(a_j)(x,t)
                     \mathcal{G}_{\mathcal{L}+\alpha,\mathbb{C}}\left(\langle b_j ,f\rangle_{\B^*,\B}\right)(x,t) \frac{dx dt}{t}.
        \end{align*}
        Since, $\langle b_j ,f\rangle_{\B^*,\B} \in BMO_\LL(\R)$, $j=1, \dots, m$, the proof can be completed by
        using \eqref{2.B}.
    \end{proof}

    We now prove \eqref{2.B.0}.  Let $f \in BMO_\LL(\R, \B)$. We denote by $\mathcal{A}$ the following linear space
    $$\mathcal{A}
        = \spann \{ a : a \text{ is a atom in } H_\LL^1(\R) \}.$$
    We have that
    $$\|f\|_{BMO_\LL(\R, \B)}
        =  \sup_{\substack{ a \in \mathcal{A} \otimes \B^*  \\  \|a\|_{H^1_\LL(\R,\B^*)} \leq 1}}
        \left| \int_\R \langle f(x), a(x) \rangle_{\B,\B^*} dx \right|.$$
    Note that, according to \cite[Lemma 2.4]{Hy2} $\mathcal{A} \otimes \B^*$ is a dense subspace of $H^1_\LL(\R,\B^*)$.
    Moreover, since $\B$ is UMD, $\B$ is reflexive and $\B^*$ is also a UMD space. Hence $\left(H^1_\LL(\R,\B^*)\right)^*=BMO_\LL(\R, \B)$.\\

    By Proposition~\ref{Prop2.1} we deduce that
    $$\int_{\R} \langle a(x) , f(x) \rangle_{\B^*,\B}dx
        = 4 \int_0^\infty \int_{\R}
                \langle \mathcal{G}_{\mathcal{L}+\alpha,\B^*}(a)(x,t) , \mathcal{G}_{\mathcal{L}+\alpha,\B}(f)(x,t) \rangle_{\B^*,\B} \frac{dx dt}{t},
        \quad a \in \mathcal{A} \otimes \B^*.$$

    \begin{Prop}\label{Prop2.2}
        Let $Y$ be a Banach space. Suppose that $g \in BMO_\LL(\R, Y)$ and $h \in H^1_\LL(\R,Y^*)$ such that
        $$\int_\R \left| \langle h(x),g(x) \rangle_{Y^*,Y}\right| dx < \infty.$$
        Then,
        $$\left| \int_\R  \langle h(x),g(x) \rangle_{Y^*,Y} dx \right|
            \leq C \|h\|_{H^1_\LL(\R,Y^*)} \|g\|_{BMO_\LL(\R, Y)} .$$
    \end{Prop}

    \begin{proof}
        Note firstly that $g$ defines an element $T_g$ of $\left(H^1_\LL(\R,Y^*)\right)^*$ such that
        $$T_g(a)
            = \int_\R  \langle a(x), g(x) \rangle_{Y^*,Y} dx,$$
        and
        $$\left| \int_\R  \langle a(x), g(x) \rangle_{Y^*,Y} dx \right|
            \leq C \|a\|_{H^1_\LL(\R,Y^*)} \|g\|_{BMO_\LL(\R, Y)} ,$$
        provided that $a$ is a linear combination of atoms in $H^1_\LL(\R,Y^*)$. Moreover, it is well-known that the function
        $F(x)=\langle a(x), g(x) \rangle_{Y^*,Y}$, $x \in \R$, might not be integrable on $\R$
        when $a \in H^1_\LL(\R,Y^*)$. On the other hand, if $\tilde{g} \in L^\infty(\R,Y)$,
        then $\tilde{g} \in BMO_\LL(\R,Y)$,
        $$T_{\tilde{g}}(a)
            = \int_\R  \langle a(x) ,  \tilde{g}(x) \rangle_{Y^*,Y} dx,
            \quad a \in H^1_\LL(\R,Y^*),$$
        and
        $$\left| \int_\R  \langle a(x) ,  \tilde{g}(x) \rangle_{Y^*,Y} dx \right|
            \leq C \|a\|_{H^1_\LL(\R,Y^*)} \|\tilde{g}\|_{BMO_\LL(\R, Y)} ,
             \quad a \in H^1_\LL(\R,Y^*).$$

        Let $\ell \in \mathbb{N}$. We define the function $\Phi: Y \longrightarrow Y$ by
        $$\Phi_\ell(b)
            = \left\{
            \begin{array}{ll}
                \dfrac{\ell b}{\|b\|_Y}, & \|b\|_Y \geq \ell, \\
                &\\
                b, & \|b\|_Y < \ell.
            \end{array}\right.$$
        $\Phi$ is a Lipschitz function. Indeed, let $b_1$, $b_2 \in Y$. If $\|b_1\|_Y \geq \ell$ and $\|b_2\|_Y \geq \ell$, then
        \begin{align*}
            \left\| \Phi_\ell(b_1) - \Phi_\ell(b_2)  \right\|_Y
                = & \left\| \dfrac{\ell b_1}{\|b_1\|_Y} - \dfrac{\ell b_2}{\|b_2\|_Y}  \right\|_Y
                \leq  \left\| b_1 - b_2 \dfrac{\|b_1\|_Y}{\|b_2\|_Y}  \right\|_Y \\
                \leq & \left\| b_1 - b_2 \right\|_Y
                     + \|b_2\|_Y \left| 1 - \dfrac{\|b_1\|_Y}{\|b_2\|_Y}  \right|
                \leq 2 \|b_1-b_2\|_Y.
        \end{align*}
        Moreover, if $\|b_1\|_Y < \ell$ and $\|b_2\|_Y \geq \ell$, it follows that
        \begin{align*}
            \left\| \Phi_\ell(b_1) - \Phi_\ell(b_2)  \right\|_Y
                = & \left\| b_1 -  \dfrac{\ell b_2}{\|b_2\|_Y}  \right\|_Y
                \leq  \left\| b_1 - b_2 \right\|_Y + \left\| b_2 -  \dfrac{\ell b_2}{\|b_2\|_Y}  \right\|_Y \\
                \leq & \left\| b_1 - b_2 \right\|_Y + \left| \|b_2\|_Y - \ell \right|
                \leq \left\| b_1 - b_2 \right\|_Y +  \|b_2\|_Y - \|b_1\|_Y
                \leq 2 \|b_1-b_2\|_Y.
        \end{align*}
        We define the function $g_\ell(x)=\Phi_\ell(g(x))$, $x \in \R$. We have that $g_\ell \in BMO_\LL(\R,Y)$ and
        $\|g_\ell\|_{BMO_\LL(\R,Y)} \leq C \|g\|_{BMO_\LL(\R,Y)} $. Moreover,
        $$\left| \langle h(x), g_\ell(x) \rangle_{Y^*,Y}\right|
            \leq \left| \langle h(x),g(x) \rangle_{Y^*,Y}\right|,
            \quad \text{a.e. } x \in \R.$$

        By using convergence dominated Theorem, since
        $\displaystyle \lim_{\ell \to \infty} \langle h(x), g_\ell(x) \rangle_{Y^*,Y} = \langle h(x),g(x) \rangle_{Y^*,Y}$
        a.e. $x \in \R$, we deduce that
        \begin{align*}
            \left| \int_\R  \langle h(x),g(x) \rangle_{Y^*,Y} dx \right|
                = & \lim_{\ell \to \infty} \left| \int_\R  \langle h(x) , g_\ell(x) \rangle_{Y^*,Y} dx \right|
                \leq  C \underset{\ell \to \infty}{\overline{\lim}}  \|h\|_{H^1_\LL(\R,Y^*)} \|g_\ell\|_{BMO_\LL(\R, Y)}\\
                \leq & C \|h\|_{H^1_\LL(\R,Y^*)} \|g\|_{BMO_\LL(\R, Y)} .
        \end{align*}
    \end{proof}

    Suppose that $a=\sum_{j=1}^m a_j b_j$, where $a_j$ is an atom for $H^1_\LL(\R)$ and $b_j \in \B^*$, $j=1, \dots, m \in \mathbb{N}$. Then,
    according to Theorem~\ref{Th1.1} for $H^1_\LL(\R,\B^*)$, we have that
    $$\mathcal{G}_{\mathcal{L}+\alpha,\B^*}(a)
        = \sum_{j=1}^m b_j \mathcal{G}_{\mathcal{L}+\alpha,\mathbb{C}}(a_j) \in H^1_\LL(\R,\gamma(H,\B^*)). $$

    If $(e_\ell)_{\ell=1}^\infty$ is an orthonormal basis in $H$ by taking into account that $\gamma(H,\B)^*=\gamma(H,\B^*)$ via trace duality
    we can write
    \begin{align*}
        & \left\langle  \mathcal{G}_{\mathcal{L}+\alpha,\B^*}(a)(x,\cdot), \mathcal{G}_{\mathcal{L}+\alpha,\B}(f)(x,\cdot)  \right\rangle_{\gamma(H,\B^*),\gamma(H,\B)} \\
        & \qquad \qquad = \sum_{j=1}^m  \left\langle b_j \mathcal{G}_{\mathcal{L}+\alpha,\mathbb{C}}(a_j)(x,\cdot), \mathcal{G}_{\mathcal{L}+\alpha,\B}(f)(x,\cdot)  \right\rangle_{\gamma(H,\B^*),\gamma(H,\B)} \\
        & \qquad \qquad = \sum_{j=1}^m \sum_{\ell =1}^\infty \int_0^\infty e_\ell(t) \int_0^\infty
                \left\langle b_j\mathcal{G}_{\mathcal{L}+\alpha,\mathbb{C}}(a_j)(x,u) , \mathcal{G}_{\mathcal{L}+\alpha,\B}(f)(x,t) \right\rangle_{\B^*,\B} e_\ell(u) \frac{du}{u } \frac{ dt}{t}
    \end{align*}
    \begin{align*}
        & \qquad \qquad = \sum_{j=1}^m \sum_{\ell =1}^\infty  \int_0^\infty e_\ell(t) \int_0^\infty
                 \mathcal{G}_{\mathcal{L}+\alpha,\mathbb{C}}(a_j)(x,u) \mathcal{G}_{\mathcal{L}+\alpha,\mathbb{C}}\left(\left\langle b_j, f \right\rangle_{\B^*,\B}\right)(x,t)  e_\ell(u) \frac{du}{u } \frac{ dt}{t} \\
        & \qquad \qquad = \sum_{j=1}^m \int_0^\infty
                 \mathcal{G}_{\mathcal{L}+\alpha,\mathbb{C}}(a_j)(x,t) \mathcal{G}_{\mathcal{L}+\alpha,\mathbb{C}}\left(\left\langle b_j, f \right\rangle_{\B^*,\B}\right)(x,t) \frac{ dt}{t} \\
        & \qquad \qquad = \sum_{j=1}^m \int_0^\infty \left\langle
                 \mathcal{G}_{\mathcal{L}+\alpha,\mathbb{B}^*}(a_jb_j)(x,t) , \mathcal{G}_{\mathcal{L}+\alpha,\mathbb{B}}\left( f\right)(x,t) \right\rangle_{\B^*,\B} \frac{ dt}{t} \\
        & \qquad \qquad = \int_0^\infty \left\langle
                 \mathcal{G}_{\mathcal{L}+\alpha,\mathbb{B}^*}(a)(x,t) , \mathcal{G}_{\mathcal{L}+\alpha,\mathbb{B}}\left( f\right)(x,t) \right\rangle_{\B^*,\B} \frac{ dt}{t},
        \quad \text{a.e. } x \in \R.
    \end{align*}
    Moreover, since $\left\langle b_j, f \right\rangle_{\B^*,\B} \in BMO_\LL(\R)$, $j=1, \dots, m$, from \eqref{2.6.1} we deduce that
    \begin{align*}
        & \int_\R \left| \left\langle  \mathcal{G}_{\mathcal{L}+\alpha,\B^*}(a)(x,\cdot), \mathcal{G}_{\mathcal{L}+\alpha,\B}(f)(x,\cdot)  \right\rangle_{\gamma(H,\B^*),\gamma(H,\B)} \right| dx \\
        & \qquad \qquad \leq \sum_{j=1}^m \int_\R \int_0^\infty
                 \left| \mathcal{G}_{\mathcal{L}+\alpha,\mathbb{C}}(a_j)(x,t) \right|
                 \left| \mathcal{G}_{\mathcal{L}+\alpha,\mathbb{C}}\left(\left\langle b_j, f \right\rangle_{\B^*,\B}\right)(x,t) \right|
                  \frac{dt dx}{t}
        < \infty.
    \end{align*}
    Hence, according to Proposition~\ref{Prop2.2} and the results proved in \eqref{2.1.2} we get
    \begin{align*}
        \left| \int_{\R} \langle a(x) , f(x) \rangle_{\B^*,\B}dx \right|
            = & 4 \left| \int_{\R}
                \langle \mathcal{G}_{\mathcal{L}+\alpha,\B^*}(a)(x,\cdot) , \mathcal{G}_{\mathcal{L}+\alpha,\B}(f)(x,\cdot) \rangle_{\gamma(H,\B^*),\gamma(H,\B)} dx \right| \\
            \leq & C  \|\mathcal{G}_{\mathcal{L}+\alpha,\B^*}(a)\|_{H^1_\LL(\R,\gamma(H,\B^*))} \|\mathcal{G}_{\mathcal{L}+\alpha,\B}(f)\|_{BMO_\LL(\R,\gamma(H,\B))}\\
            \leq & C  \|a\|_{H^1_\LL(\R,\B^*)} \|\mathcal{G}_{\mathcal{L}+\alpha,\B}(f)\|_{BMO_\LL(\R,\gamma(H,\B))}.
    \end{align*}
    We conclude that
    $$\|f\|_{BMO_\LL(\R,\B)}
        \leq C \|\mathcal{G}_{\mathcal{L}+\alpha,\B}(f)\|_{BMO_\LL(\R,\gamma(H,\B))}.$$

%\newpage
%%%%%%%%%%%%%%%%%
    \subsection{}\label{subsec:2.4}

    We are going to show that, for every $g \in H^1_\LL(\R,\B)$,
    \begin{equation}\label{2.11}
        \|g\|_{H^1_\LL(\R,\B)}
            \leq C \|\mathcal{G}_{\mathcal{L}+\alpha,\B}(g)\|_{H^1_\LL(\R,\gamma(H,\B))}.
    \end{equation}
    Suppose that $a \in \mathcal{A} \otimes \B$, where $\mathcal{A}$ is defined in Section~\ref{subsec:2.3}. Since $\B$ is UMD,
    $\left( H^1_\LL(\R,\B) \right)^*=BMO_\LL(\R,\B^*)$, and we have that
    $$\|a\|_{H^1_\LL(\R, \B)}
        =  \sup_{\substack{ f \in BMO_\LL(\R,\B^*)  \\  \|f\|_{BMO_\LL(\R,\B^*)} \leq 1}}
        \left| \int_\R \langle f(x), a(x) \rangle_{\B^*,\B} dx \right|.$$
    Moreover, for every $f \in BMO_\LL(\R,\B^*)$, since $\mathcal{G}_{\mathcal{L}+\alpha,\B^*}$ is bounded from $BMO_\LL(\R,\B^*)$
    into $BMO_\LL(\R,\gamma(H,\B^*))$ (see Subsection~\ref{subsec:2.1}), again by Proposition~\ref{Prop2.2} it follows that
    \begin{align*}
        \left| \int_{\R} \langle f(x) , a(x) \rangle_{\B^*,\B}dx \right|
            \leq & C \|\mathcal{G}_{\mathcal{L}+\alpha,\B^*}(f)\|_{BMO_\LL(\R,\gamma(H,\B^*))} \|\mathcal{G}_{\mathcal{L}+\alpha,\B}(a)\|_{H^1_\LL(\R,\gamma(H,\B))} \\
            \leq & C \|f\|_{BMO_\LL(\R,\B^*)} \|\mathcal{G}_{\mathcal{L}+\alpha,\B}(a)\|_{H^1_\LL(\R,\gamma(H,\B))}.
    \end{align*}
    Hence,
    $$\|a\|_{H^1_\LL(\R,\B)}
            \leq C \|\mathcal{G}_{\mathcal{L}+\alpha,\B}(a)\|_{H^1_\LL(\R,\gamma(H,\B))}.$$
    Since $\mathcal{A} \otimes \B$ is a dense subspace in $H^1_\LL(\R,\B)$ and $\mathcal{G}_{\mathcal{L}+\alpha,\B}$ is bounded
    from $H^1_\LL(\R,\B)$ into $H^1_\LL(\R,\gamma(H,\B))$ (see Subsection~\ref{subsec:2.2}) we conclude that \eqref{2.11} holds for every $g \in H^1_\LL(\R,\B)$.

%\newpage
    %%%%%%%%%%%%%%%%%%%%%%%%%%%%%%%%%%%%%%%%%%%%%%%%%%%%%%%%%%%%%%%%%%%%%%%%%%%%%%%%%%%%%%%%%%%%%%%%%%%%%%%%%%%%%%%%%%%%
    \section{Proof of Theorem~\ref{Th1.2}}  \label{sec:Th1.2}
    %%%%%%%%%%%%%%%%%%%%%%%%%%%%%%%%%%%%%%%%%%%%%%%%%%%%%%%%%%%%%%%%%%%%%%%%%%%%%%%%%%%%%%%%%%%%%%%%%%%%%%%%%%%%%%%%%%%%

%\newpage
%%%%%%%%%%%%%%%%%
    \subsection{}\label{subsec:3.1}
    We are going to prove that  the operator $T_{j,+}^\LL$ is bounded from
    $BMO_\LL(\R,\B)$ into $BMO_\LL(\R,\gamma(H,\B))$. The corresponding property
    for  $T_{j,-}^\LL$ when $n \geq 3$ can be shown in a similar way.\\

    We consider the function $\Omega$ defined by
    $$\Omega(x,y,t)
        = \frac{t^2}{\sqrt{4 \pi}} \int_0^\infty \frac{e^{-t^2/4s}}{s^{3/2}} ( \partial_{x_j} + x_j) W_s^\LL (x,y) ds,
        \quad x,y \in \R \text{ and } t>0.$$
    We have that
    $$ ( \partial_{x_j} + x_j) W_s^\LL (x,y)
        = \left( x_j - \frac{1}{2} \frac{1+e^{-2s}}{1-e^{-2s}}(x_j-y_j) - \frac{1}{2} \frac{1-e^{-2s}}{1+e^{-2s}}(x_j+y_j)  \right) W_s^\LL (x,y) ,
        \quad x,y \in \R \text{ and } s>0.$$
    Note that $|a| \leq |a+b|+|a-b|$, $a,b \in \mathbb{R}$. Then, it follows that, for every $x,y \in \R$ and $s>0$,
    \begin{align*}
        \left| ( \partial_{x_j} + x_j) W_s^\LL (x,y) \right|
            \leq C \frac{1}{\sqrt{1-e^{-2s}}} \left( \frac{e^{-2s}}{1-e^{-4s}} \right)^{n/2}
                \exp \left( -\frac{1}{8} \left( \frac{1+e^{-2s}}{1-e^{-2s}}|x-y|^2 + \frac{1-e^{-2s}}{1+e^{-2s}}|x+y|^2  \right) \right).
    \end{align*}
   As in \eqref{2.1} we obtain, for every $x,y \in \R$ and $t>0$,
 \begin{align}\label{Omega}
|\Omega(x,y,t)|
        \leq &C t^2 e^{-c(|x-y|^2+|y| |x-y|)}
            \int_0^\infty \frac{e^{-c(t^2+|x-y|^2)/s}e^{-ns}}{s^{3/2} (1-e^{-4s})^{(n+1)/2}}ds\nonumber\\
    \leq &C t^2 e^{-c(|x-y|^2+|y| |x-y|)}
            \int_0^\infty \frac{e^{-c(t^2+|x-y|^2)/s}}{s^{(n+4)/2} } ds\leq C\frac{t^2e^{-c(|x-y|^2+|y| |x-y|)}}{(t+|x-y|)^{n+2}}.
\end{align}
    Hence, it follows that
    \begin{align}\label{N2.0}
        \|\Omega(x,y,\cdot)\|_H
            \leq & C e^{-c(|x-y|^2+|y| |x-y|)}
                \left(\int_0^\infty \frac{t^3}{ (t+|x-y|)^{2n+4}}dt\right)^{1/2}\nonumber\\
    \leq &C\frac{e^{-c(|x-y|^2+|y| |x-y|)}}{|x-y|^n},
            \quad x,y \in \R, \ x \neq y.
    \end{align}

    Let $i=1, \dots, n$. We can write, if $i \neq j$,
    \begin{align*}
        \partial_{x_i}( \partial_{x_j} + x_j) W_s^\LL (x,y)
            = & - \left( x_j - \frac{1}{2} \frac{1+e^{-2s}}{1-e^{-2s}}(x_j-y_j) - \frac{1}{2} \frac{1-e^{-2s}}{1+e^{-2s}}(x_j+y_j)  \right) \\
            & \times      \left(\frac{1}{2} \frac{1+e^{-2s}}{1-e^{-2s}}(x_i-y_i) + \frac{1}{2} \frac{1-e^{-2s}}{1+e^{-2s}}(x_i+y_i)  \right)W_s^\LL (x,y) ,
        \quad x,y \in \R \text{ and } s>0,
    \end{align*}
    and
    \begin{align*}
        \partial_{x_j}( \partial_{x_j} + x_j) W_s^\LL (x,y)
            = &  -\Big\{ \frac{2 e^{-4s}}{1-e^{-4s}} + \left( x_j- \frac{1}{2} \frac{1+e^{-2s}}{1-e^{-2s}}(x_j-y_j) - \frac{1}{2} \frac{1-e^{-2s}}{1+e^{-2s}}(x_j+y_j) \right) \\
            & \times \left(\frac{1}{2} \frac{1+e^{-2s}}{1-e^{-2s}}(x_j-y_j) + \frac{1}{2} \frac{1-e^{-2s}}{1+e^{-2s}}(x_j+y_j)  \right)\Big\}W_s^\LL (x,y) ,
        \quad x,y \in \R \text{ and } s>0.
    \end{align*}
    Then, we get, for each $x,y \in \R$ and $s>0$,
    $$|\partial_{x_i}( \partial_{x_j} + x_j) W_s^\LL (x,y)|
        \leq C \frac{1}{1-e^{-2s}} \left( \frac{e^{-2s}}{1-e^{-4s}} \right)^{n/2}
                \exp \left( -\frac{1}{8} \left( \frac{1+e^{-2s}}{1-e^{-2s}}|x-y|^2 + \frac{1-e^{-2s}}{1+e^{-2s}}|x+y|^2  \right) \right).$$
    By proceeding as above we obtain
    \begin{align}\label{N3.1}
        \|\partial_{x_i} \Omega(x,y,\cdot)\|_H
            \leq &  \frac{C}{|x-y|^{n+1}},
            \quad x,y \in \R, \ x \neq y.
    \end{align}
    In a similar way we can see that
    \begin{align}\label{N3.2}
        \|\partial_{y_i} \Omega(x,y,\cdot)\|_H
            \leq &  \frac{C}{|x-y|^{n+1}},
            \quad x,y \in \R, \ x \neq y.
    \end{align}
    Putting together \eqref{N3.1} and \eqref{N3.2} we conclude that
    \begin{align*} %\label{N3.3}
        \| \nabla_x \Omega(x,y,\cdot)\|_H +  \| \nabla_y \Omega(x,y,\cdot)\|_H
            \leq &  \frac{C}{|x-y|^{n+1}},
            \quad x,y \in \R, \ x \neq y.
    \end{align*}

    According to \cite[Theorem 2]{BCCFR} the operator $T^\LL_{j,+}$ is bounded from $L^2(\R,\B)$ into
    $L^2(\R, \gamma(H,\B))$. Moreover, the same argument we have used in Subsection~\ref{subsec:2.1} allows us to show that, for every
    $f \in L^\infty_c(\R,\B)$,
    $$T^\LL_{j,+}(f)(x,t)
        = \left( \int_\R \Omega (x,y, \cdot) f(y) dy \right)(t),
        \quad \text{a.e. } x \notin \supp(f). $$
    By taking into account \eqref{2.1.0}, for each $x \in \R$ and $s>0$,  we obtain that
    $$ (\partial_{x_j} + x_j) W_s^\mathcal{L}(1)(x)
            = \frac{1}{\pi^{n/2}} \left( \frac{e^{-2s}}{1+e^{-4s}} \right)^{n/2} \left( 1-\frac{1-e^{-4s}}{1+e^{-4s}} \right)x_j
                \exp\left(-\frac{1-e^{-4s}}{2(1+e^{-4s})}|x|^2\right). $$
    Hence, Minkowski's inequality leads to
    \begin{align*}
        \| T^\LL_{j,+}(1)(x,\cdot) \|_H
            \leq & C \int_0^\infty \frac{e^{-s}}{\sqrt{s}} \left\| t (\partial_{x_j} + x_j) W_{t^2/4s}^\mathcal{L}(1)(x)  \right\|_H ds \\
            \leq & C \int_0^\infty e^{-s}\left\| \sqrt{u} (\partial_{x_j} + x_j) W_{u}^\mathcal{L}(1)(x)  \right\|_H ds
            \leq C, \quad  x \in \R.
    \end{align*}
    In a similar way we can see that $\nabla_x T^\LL_{j,+}(1) \in L^\infty(\R,H)$.\\

    By using Theorem~\ref{Th1.0} we conclude that $T^\LL_{j,+}$  is bounded from $BMO_\LL(\R,\B)$ into $BMO_\LL(\R,\gamma(H,\B))$.

%\newpage
%%%%%%%%%%%%%%%%%
    \subsection{}\label{subsec:3.2}

    We are going to see that $T^\LL_{j,+}$ is a bounded operator from $H^1_\LL(\R,\B)$ into $H^1_\LL(\R,\gamma(H,\B))$.
    The boundedness property of $T^\LL_{j,-}$ can be proved  in a similar way, for $n \geq 3$.\\

    In Subsection~\ref{subsec:3.1} we saw that $T^\LL_{j,+}$ is a Calder\'on-Zygmund operator. Hence, it follows that
    $T^\LL_{j,+}$ can be extended from $L^2(\R,\B)\cap L^1(\R,\B)$ to $L^1(\R,\B)$ as a bounded operator from
    $H^1(\R,\B)$ into $L^1(\R,\gamma(H,\B))$ and from $L^1(\R,\B)$ into $L^{1,\infty}(\R,\gamma(H,\B))$.
    Moreover, according to \cite[Theorem 2]{BCCFR}, $T^\LL_{j,+}$ is a bounded operator from
    $H^1(\R,\B)$ into $L^1(\R,\gamma(H,\B))$ and from $L^1(\R,\B)$ into $L^{1,\infty}(\R,\gamma(H,\B))$.
    By using \eqref{N2.0}, the procedure developed in Subsection~\ref{subsec:2.2} allows us to see that the operator $T^\LL_{j,+}$
    is bounded from $H^1_\LL(\R,\B)$ into $L^1(\R,\gamma(H,\B))$.\\

    We consider the maximal operator $S$ defined by
    $$S(f)(x)
        = \sup_{s>0} \left\| P_s^{\LL+2} \left( T_{j,+}^\LL (f) \right)(x,\cdot)\right\|_{\gamma(H,\B)}.$$
    According to Proposition~\ref{Prop2.4} the proof of our objective will be finished when we establish that
    the operator $S$ is bounded from $H^1_\LL(\R,\B)$ into $L^1(\R)$.\\

    The maximal operator $\mathcal{M}_*$ given by
    $$\mathcal{M}_*(g)
        = \sup_{s>0} \| P_s^{\LL+2}(g) \|_{\gamma(H,\B)}$$
    is known to be bounded from $L^p(\R,\gamma(H,\B))$ into $L^p(\R)$, for every $1<p<\infty$, and from $L^1(\R,\gamma(H,\B))$ into $L^{1,\infty}(\R)$.
    Since $T_{j,+}^\LL$ is bounded from $L^p(\R,\B)$ into $L^p(\R,\gamma(H,\B))$, $1 < p <\infty$, from
    $L^1(\R,\B)$ into $L^{1,\infty}(\R,\gamma(H,\B))$, and from $H^1(\R,\B)$ into $L^{1}(\R,\gamma(H,\B))$,
    the operator $\mathbb{S}$ defined by
    $$\mathbb{S}(f)(x,s,t)
        = P_s^{\LL+2} \left( T_{j,+}^\LL (f)(\cdot, t) \right) (x) $$
    is bounded from $L^p(\R,\B)$ into $L^p(\R,L^\infty((0,\infty),\gamma(H,\B)))$, $1 < p <\infty$
    and from $H^1(\R,\B)$ into $L^{1,\infty}(\R,L^\infty((0,\infty),\gamma(H,\B)))$.\\

    According to \cite[Lemmas 4.1 and 4.2]{ST1} we have that, for every $f \in L^\infty_c(\R) \otimes \B$,
    $$\mathbb{S}(f)(x,s,t)
        = t (\partial_{x_j} + x_j) P_{s+t}^{\LL} (f)(x),
        \quad x \in \R \text{ and } s,t >0. $$

    We consider the function
    $$\Y(x,y,s,t)
        = t \frac{s+t}{\sqrt{4 \pi}} \int_0^\infty \frac{e^{-(s+t)^2/4u}}{u^{3/2}} (\partial_{x_j} + x_j) W_u^\LL(x,y) du,
        \quad x,y \in \R, \ x \neq y \text{ and } s,t>0. $$
    By proceeding as in \eqref{N2.0} we can see that
    \begin{align} \label{N3.4}
        \| \Y(x,y,\cdot,\cdot)\|_{L^\infty((0,\infty),H)}
            \leq & C \frac{e^{-c(|x-y|^2+|y| |x-y|)}}{|x-y|^{n}},
            \quad x,y \in \R, \ x \neq y,
    \end{align}
    and
    \begin{align*}
        \| \nabla_x\Y(x,y,\cdot,\cdot)\|_{L^\infty((0,\infty),H)} + \| \nabla_y \Y(x,y,\cdot,\cdot)\|_{L^\infty((0,\infty),H)}
            \leq &  \frac{C}{|x-y|^{n+1}},
            \quad x,y \in \R, \ x \neq y.
    \end{align*}

    Moreover, as in Subsection~\ref{subsec:2.2} we can see that, for every $g \in L^\infty_c(\R) \otimes \B$,
    $$\mathbb{S}(g)(x,s,t)
        = \left( \int_\R \mathcal{Y}(x,y,\cdot,\cdot) g(y) dy \right)(s,t),
        \quad x \notin \supp(g),$$
    being the integral understood in the $L^\infty((0,\infty),\gamma(H,\B))$-Bochner sense.\\

    Vector valued Calder\'on-Zygmund theory implies that the operator $\mathbb{S}$ can be extended from
    $L^2(\R,\B) \cap L^1(\R,\B)$ to $L^1(\R,\B)$ as a bounded operator from
    $L^1(\R,\B)$ to $L^{1,\infty}(\R,L^\infty((0,\infty),$ $\gamma(H,\B)))$ and from
    $H^1(\R,\B)$ into $L^{1}(\R,L^\infty((0,\infty),\gamma(H,\B)))$. In order to see that
    $\mathbb{S}$ is in fact bounded from $L^1(\R,\B)$ into $L^{1,\infty}(\R,L^\infty((0,\infty),\gamma(H,\B)))$
    and from $H^1(\R,\B)$ into $L^{1}(\R,$ $L^\infty((0,\infty),\gamma(H,\B)))$,
    we can proceed as in the end of the proof of Proposition~\ref{Prop2.5}.\\

    By taking into account that
    \begin{itemize}
        \item \eqref{N3.4} holds,
        \item $\mathbb{S}$ is bounded from $L^1(\R,\B)$ into $L^{1,\infty}(\R,L^\infty((0,\infty),\gamma(H,\B)))$,
        \item $\mathbb{S}$ is bounded from $H^1(\R,\B)$ into $L^{1}(\R,L^\infty((0,\infty),\gamma(H,\B)))$,
    \end{itemize}
    we can prove, by using the procedure employed in the final part of Subsection~\ref{subsec:2.2},
    that $\mathbb{S}$ is bounded from $H^1_\LL(\R,\B)$ into $L^{1}(\R,L^\infty((0,\infty),\gamma(H,\B)))$.\\

    Thus the proof of Theorem~\ref{Th1.2} for $T_{j,+}^\LL$ is finished.

%\newpage
    %%%%%%%%%%%%%%%%%%%%%%%%%%%%%%%%%%%%%%%%%%%%%%%%%%%%%%%%%%%%%%%%%%%%%%%%%%%%%%%%%%%%%%%%%%%%%%%%%%%%%%%%%%%%%%%%%%%%
    \section{Proof of Theorem~\ref{Th1.3}}  \label{sec:Th1.3}
    %%%%%%%%%%%%%%%%%%%%%%%%%%%%%%%%%%%%%%%%%%%%%%%%%%%%%%%%%%%%%%%%%%%%%%%%%%%%%%%%%%%%%%%%%%%%%%%%%%%%%%%%%%%%%%%%%%%%

    Theorems \ref{Th1.1} and \ref{Th1.2} show that $(i)$ implies  $(ii)$ and $(i)$ implies  $(iii)$.\\

    Suppose that  $(ii)$ is true for some $j=1,\dots,n.$ Let $f=\sum_{i=1}^m f_i b_i,$ where
    $f_i\in H^1_\LL(\R)$ and $b_i\in \B,$ $i=1,\dots,m\in \mathbb{N}.$ We denote by $R_{j,+}^\LL$ the $j$-th Riesz transform in the
    Hermite setting (see Appendix for definitions).  According to Proposition \ref{Pr3.2},
    $$R_{j,+}^\LL(f)=\sum_{i=1}^m b_i R_{j,+}^\LL (f_i) \in H_\LL^1(\R)\otimes\B.$$

    By applying \cite[Lemmas 4.1 and 4.2]{ST1} we get, for every atom $a$ for $H^1_\LL(\R)$,
    $$T_{j,+}^\LL(a)=-\mathcal{G}_{\LL+2,\mathbb{C}} R_{j,+}^\LL(a).$$
    Moreover, $T_{j,+}^\LL$ and $\mathcal{G}_{\LL+2,\mathbb{C}} \circ R_{j,+}^\LL$ are bounded operators from $H^1_\LL(\R)$ into $H^1_\LL(\R,H)$
    (see Theorem~\ref{Th1.2}, Proposition~\ref{Pr3.2} and Theorem~\ref{Th1.1}). Then, we have that
    $$T_{j,+}^\LL(g)=-\mathcal{G}_{\LL+2,\mathbb{C}} R_{j,+}^\LL(g), \quad g \in H^1_\LL(\R),$$
    and this implies
    $$T_{j,+}^\LL(f)=-\mathcal{G}_{\LL+2,\B} R_{j,+}^\LL(f).$$

    We can write
    \begin{align*}
        \|R_{j,+}^\LL(f)\|_{H^1_\LL(\R,\B)}
            \leq &C\|\mathcal{G}_{\LL+2,\B} R_{j,+}^\LL (f)\|_{H^1_\LL(\R,\gamma(H,B))}
            =  C\|T_{j,+}^\LL f \|_{H^1_\LL(\R, \gamma(H,\B))}
            \leq C\|f\|_{H^1_\LL(\R,\B)}.
    \end{align*}
    Since $H^1_\LL(\R)\otimes \B$ is a dense subspace of $H^1_\LL(\R,\B)$ (\cite[Lemma 2.4]{Hy2}), $H^1(\R,\B) \subset H^1_\LL(\R,\B)$
    and $H^1_\LL(\R,\B) \subset L^1(\R,\B)$, \cite[Theorem 4.1]{MTX}
    implies that $R_{j,+}^\LL$ can be extended to $L^2(\R,\B)$ as a bounded operator from $L^2(\R,\B)$ into itself.
    Then, from \cite[Theorem 2.3]{AT} we deduce that $\B$ is UMD.\\

    Assume now $(iii)$ holds for some $j=1,\dots, n.$ By proceeding as above, this time applying Proposition~\ref{Pr3.1},
    we can see that, for every $f\in L^\infty_c(\R)\otimes \B,$
    \begin{equation}\label{r3.8}
    \|R_{j,+}^\LL (f) \|_{BMO_\LL(\R, \B)}\leq C\|f\|_{BMO_\LL(\R,\B)}.
    \end{equation}

    Let $\E$ be a finite dimensional subspace of $\B.$ By taking into account that
    $L^\infty_c(\R)\otimes \E= L^\infty_c(\R,\E) \subset BMO_\LL(\R,\E)$ and $BMO_\LL(\R,\E) \subset BMO(\R,\E)$, from \eqref{r3.8} and
    \cite[Theorem 4.1]{MTX} we deduce that $R_{j,+}^\LL$ can be extended to $L^2(\R,\E)$ as a bounded operator from $L^2(\R, \E)$ into itself and
    $$ \| R_{j,+}^\LL(f) \|_{L^2(\R,\E)}\leq C\|f\|_{L^2(\R,\E)}, \quad f\in L^2(\R,\E),$$
    where $C>0$ does not depend on $\E.$ Hence,
    $$ \|R_{j,+}^\LL(f)\|_{L^2(\R,\B)}\leq C\|f\|_{L^2(\R,\B)}, \quad f\in L^2(\R)\otimes \B.$$
    From \cite[Theorem 2.3]{AT} it follows that $\B$ is UMD.\\

    The proof of the result when $T_{j,+}^\LL$ and $\mathcal{G}_{\LL + 2,\B}$ are replaced by $T_{j,-}^\LL$
    and $\mathcal{G}_{\LL-2, \B}$, respectively, can be made similarly, for every $n \geq 3$.

%\newpage
    %%%%%%%%%%%%%%%%%%%%%%%%%%%%%%%%%%%%%%%%%%%%%%%%%%%%%%%%%%%%%%%%%%%%%%%%%%%%%%%%%%%%%%%%%%%%%%%%%%%%%%%%%%%%%%%%%%%%
    \section{Appendix}  \label{sec:app}
    %%%%%%%%%%%%%%%%%%%%%%%%%%%%%%%%%%%%%%%%%%%%%%%%%%%%%%%%%%%%%%%%%%%%%%%%%%%%%%%%%%%%%%%%%%%%%%%%%%%%%%%%%%%%%%%%%%%%

    The Hermite operator $L$ can be written as follows
    $$ L=-\frac{1}{2}[(\nabla + x)(\nabla - x)+(\nabla -x)(\nabla + x)].$$
    This decomposition suggests to call Riesz transforms in the Hermite setting to the operators formally defined by
    \begin{equation}\label{3.1}
    R_{j,\pm}^\LL= \left( \partial_{x_j} \pm x_j \right)\LL^{-1/2}, \quad j=1,\dots, n.
    \end{equation}
    (see \cite{ST1} and \cite{Th2}).\\

    Let $j=1, \dots,n.$ We denote by $e_j$ the $j-$th coordinate vector in $\R.$ It is well known that
    \begin{equation}\label{3.2}
    \left(\partial_{x_j}+x_j\right)h_k=(2k_j)^{1/2}h_{k-e_j}, \quad \left( \partial_{x_j}-x_j\right)h_k= -(2k_j+2)^{1/2}h_{k+e_j},
    \end{equation}
    for every $k=(k_1,\dots, k_n)\in \mathbb{N}^n$.\\

    The negative square root $\LL^{-1/2}$ of $\LL$ is defined by
    $$ \LL^{-1/2}(f)(x)
        =\int_0^\infty P_t^\LL(f)(x) dt, \quad f\in L^2(\R).$$
    We have that
    \begin{equation}\label{3.3}
    \LL^{-1/2}(f)= \sum_{k\in \mathbb{N}^n}\frac{1}{\sqrt{2|k|+n}}\langle f, h_k \rangle h_k, \quad f\in L^2(\R).
    \end{equation}
    Equalities \eqref{3.1}, \eqref{3.2} and \eqref{3.3} lead to define the Riesz transforms $R_{j,\pm}^\LL$ by
    $$R_{j,+}^\LL(f)=\sum_{k\in \mathbb{N}^n} \sqrt{\frac{2k_j}{2|k|+n}} \langle f, h_k \rangle h_{k-e_j}, \quad f\in L^2(\R),$$
    and
    $$R_{j,-}^\LL(f)=-\sum_{k\in \mathbb{N}^n} \sqrt{\frac{2k_j+2}{2|k|+n}} \langle f, h_k \rangle h_{k+e_j}, \quad f\in L^2(\R).$$
    Plancherel equality imply that $R_{j,\pm}^\LL$ is bounded from $L^2(\R)$ into itself. $L^p-$boundedness properties of $R_{j,\pm}^\LL$ were
    established by Stempak and Torrea in \cite{ST1} (see also \cite{Th1}). They use Calder\'on-Zygmund theory and show that $R_{j,\pm}^\LL$
    are singular integrals associated to the Calder\'on-Zygmund kernels
    \begin{equation}\label{Rkern}
        R_{j,\pm}^\LL(x,y)= \frac{1}{\sqrt{\pi}}\int_0^\infty \left( \partial_{x_j}\pm x_j \right) P_t^\LL(x,y)  dt, \quad x,y\in \R, \ x\neq y.
    \end{equation}
    $R_{j,\pm}^\LL$ can be extended from $L^2(\R)\cap L^p(\R)$ to $L^p(\R)$ as a bounded operator from
    $L^p(\R)$ into itself, $1<p<\infty,$ and from $L^1(\R)$ into $L^{1,\infty}(\R)$ (\cite[Corollary 3.4]{ST1}).
    We continue denoting by $R_{j,\pm}^\LL$ the extended operators.\\

    In the following propositions we analyze the behavior of $R_{j,\pm}^\LL$, $j=1, \dots, n$ in the spaces $BMO_\LL(\R)$ and
    $H^1_\LL(\R)$.

    \begin{Prop}\label{Pr3.1}
        Let $j=1,\dots, n.$ Then, the Riesz transforms $R_{j,\pm}^\LL$ are bounded from $BMO_\LL(\R)$ into itself.
    \end{Prop}
    \begin{proof}
        We only analyze  $R_{j,+}^\LL.$ The operator $R_{j,-}^\LL$ can be studied similarly.
        In \cite[Section 4.3]{BCFST} it was shown that the operator $R_{j,+}^\LL-x_j\LL^{-1/2}$
        is bounded from $BMO_\LL(\R)$ into itself.\\

        We consider now the operator $T_j=x_j \LL^{-1/2}.$ By \eqref{1.3.1} we can write
        $$ T_j(f)(x)=\frac{x_j}{\sqrt{\pi}} \int_0^\infty W_t^\LL (f)(x) \frac{dt}{\sqrt{t}}= \int_0^\infty M_j(x,y)f(y)dy, \quad f\in L^2(\R),$$
        where
        $$ M_j(x,y)= \frac{x_j}{\sqrt{\pi}} \int_0^\infty W_t^\LL(x,y)\frac{dt}{\sqrt{t}}, \quad x,y\in \R, \ x\neq y.$$
        According to \cite[Lemma 3]{BT} the operator $T_j$ is bounded from $L^2(\R)$
        into itself.\\

        We are going to show that
        \begin{equation*}\label{3.7}
        |M_j(x,y)|\leq C\frac{e^{-c(|x-y|^2+|x||x-y|)}}{|x-y|^n}, \quad x,y\in \R \text{ and } x\neq y,
        \end{equation*}
        and
        \begin{equation*}\label{3.8}
        |\nabla_x M_j(x,y)|+|\nabla_y M_j(x,y)| \leq \frac{C}{|x-y|^{n+1}}, \quad x,y\in \R \text{ and } x\neq y.
        \end{equation*}
  By using \eqref{2.1} we deduce
        \begin{align*}
        |M_j(x,y)|&\leq C |x|e^{-c(|x-y|^2+|x||x-y|)}\int_0^\infty \frac{e^{-c(|x-y|^2/t + (1-e^{2t})|x+y|^2)}}{t^{(n+1)/2}}dt \\
        &\leq C (|x+y|+|x-y|)e^{-c(|x-y|^2+|x||x-y|)}\left(\int_0^1 \frac{e^{-c(|x-y|^2/t + t|x+y|^2)}}{t^{(n+1)/2}}dt + e^{-c|x+y|^2} \right)\\
        &\leq C e^{-c(|x-y|^2+|x||x-y|)}\left( \int_0^1 \frac{e^{-c|x-y|^2/t}}{t^{(n+2)/2}}dt +1\right)
        \leq C \frac{e^{-c(|x-y|^2+|x||x-y|)}}{|x-y|^n}, \ x,y\in \R, x\neq y.
        \end{align*}
        Let $i=1,\dots,n.$ For every $x,y\in \R$, $x\neq y$, we have that
        \begin{align*}
            \partial_{x_i}   M_j(x,y)
                =& -\frac{1}{2\sqrt{\pi}}x_j \int_0^\infty \left(\frac{1+e^{-2t}}{1-e^{-2t}}(x_i-y_i)+ \frac{1-e^{-2t}}{1+e^{-2t}}(x_i+y_i)\right) W_t^\LL(x,y) \frac{dt}{\sqrt{t}}, \quad i\neq j,
        \end{align*}
        and
        \begin{align*}
            \partial_{x_j}M_j(x,y)=&\frac{1}{\sqrt{\pi}}\int_0^\infty  \left( 1- x_j \frac{1+e^{-2t}}{2(1-e^{-2t})}(x_j-y_j)- x_j \frac{1-e^{-2t}}{2(1+e^{-2t})}(x_j+y_j) \right) W_t^\LL(x,y) \frac{dt}{\sqrt{t}}.
        \end{align*}
   Hence, by (\ref{2.1}) we get
        \begin{align*}
        \left| \partial_{x_i} M_j(x,y)\right|\leq &C\int_0^\infty \left(1+|x|\frac{|x-y|}{1-e^{-2t}}+(|x+y|+|x-y|)(1-e^{-2t})|x+y|\right) W_t^\LL(x,y)\frac{dt}{\sqrt{t}}\\
        \leq & C\int_0^\infty \frac{e^{-c|x-y|^2/t}e^{-ct}}{(1-e^{-4t})^{(n+2)/2}}\frac{dt}{\sqrt{t}}\leq C\int_0^\infty \frac{e^{-c|x-y|^2/t}}{t^{(n+3)/2}}dt\\
        \leq & \frac{C}{|x-y|^{n+1}}, \quad x,y\in \R, \ x\neq y.
        \end{align*}
        In a similar way we can see that, for every $i=1,\dots, n,$
        $$ \left| \partial_{y_i} M_j(x,y)\right|\leq \frac{C}{|x-y|^{n+1}}, \quad x,y\in \R, \ x\neq y.$$

        According to \eqref{2.1.0} we can write
        $$ T_j(1)(x)=\frac{x_j}{\pi^{(n+1)/2}}\int_0^\infty \left(\frac{e^{-2t}}{1+e^{-4t}} \right)^{n/2} \exp \left( -\frac{1-e^{-4t}}{2(1+e^{-4t})}|x|^2 \right) \frac{dt}{\sqrt{t}}, \quad x\in \R. $$
        It follows that
        $$ |T_j(1)(x)| \leq C|x|\left( \int_0^1 \frac{e^{-ct|x|^2}}{\sqrt{t}}dt+ e^{-c|x|^2}\int_1^\infty e^{-nt} dt \right) \leq C, \quad x\in \R.$$
        Moreover, for every $i=1,\dots,n,$ $ i\neq j,$ we have that
        $$ \partial_{x_i}T_j(1)(x)=-\frac{x_jx_i}{\sqrt{\pi}} \int_0^\infty \frac{1-e^{-4t}}{1+e^{-4t}} W_t^\LL(1)(x) \frac{dt}{\sqrt{t}}, \quad x\in \R,$$
        and
        $$ \partial_{x_j}T_j(1)(x)= \frac{1}{\sqrt{\pi}} \int_0^\infty \left(1-\frac{1-e^{-4t}}{1+e^{4t}}x_j^2\right) W_t^\LL(1)(x) \frac{dt}{\sqrt{t}}, \quad x\in \R.$$
        Then, we can deduce that $\nabla T_j(1)\in L^\infty(\R).$\\

        By \cite[Theorem 1.1]{BCFST} we conclude that $T_j$ can be extended to $BMO_\LL(\R)$ as a bounded operator from $BMO_\LL(\R)$ into itself.
    \end{proof}

    \begin{Prop}\label{Pr3.2}
        Let $j=1,\dots, n.$ Then, the Riesz transforms $R_{j,\pm}^\LL$ can be extended from $L^2(\R) \cap H^1_\LL(\R)$
        to $H^1_\LL(\R)$ as bounded operators from $H^1_\LL(\R)$ into itself.
    \end{Prop}
    \begin{proof}
        We study the operator $R_{j,+}^\LL.$ $R_{j,-}^\LL$ can be analyzed in a similar way.\\

        By taking in mind Proposition~\ref{Prop2.4} it is enough to see that
        $R_{j,+}^\LL$ can be extended as a bounded operator from $H_\LL^1(\R)$ into $L^1(\R)$
        and that the operator $G$  defined by
        $$ G(f)(x,t)=P_t^{\LL+2}(R_{j,+}^\LL(f))(x),$$
        is bounded from $H_\LL^1(\R)$ into $L^1(\R,L^\infty(0,\infty))$.\\

        In \cite[Theorem 3.3]{ST1} it was proved that the Riesz transform $R_{j,+}^\LL$ is a Calder\'on-Zygmund operator associated to the kernel
        given in \eqref{Rkern}. Thus, in order to see that $R_{j,+}^\LL$ can be extended as a bounded operator from $H_\LL^1(\R)$ into $L^1(\R)$,
        we only need to show that
        \begin{equation}\label{Hbound}
            | R_{j,+}^\LL(x,y) |
                \leq C \frac{e^{-c(|x-y|^2 + |y||x-y|)}}{|x-y|^n}, \quad x,y \in \R, \ x \neq y.
        \end{equation}
        and then reasoning as in the end of Subsection~\ref{subsec:2.2}. Estimation (\ref{Hbound}) follows from (\ref{Omega}).\\

        Now we establish that $G$ can be extended as a bounded operator from $H_\LL^1(\R)$ into $L^1(\R,L^\infty(0,\infty))$. We observe that
        \begin{align*}
        G(f)(x,t)=&\sum_{k\in \mathbb{N}^n} \sqrt{\frac{2k_j}{2|k|+n}} e^{-t\sqrt{2|k|+n}}\langle f, h_k \rangle h_{k-e_j}(x)
        =R_{j,+}^\LL(P^\LL_t(f))(x), \quad f\in L^2(\R).
        \end{align*}

        We consider the function
        $$ \mathcal{C}(x,y,t)=\frac{1}{\sqrt{\pi}}\int_0^\infty \left( \partial_{x_j}+x_j\right)P_{t+s}^\LL(x,y)ds, \quad x,y\in \R \text{ and } t\in (0,\infty).$$
        This function $\mathcal{C}$ satisfies the following $L^\infty(0,\infty)-$Hermite-Calder\'on-Zygmund conditions:
        \begin{equation}\label{3.5}
        \|\mathcal{C}(x,y,\cdot)\|_{L^\infty(0,\infty)}\leq C \frac{e^{-c( |x-y|^2 +   |y||x-y|)}}{|x-y|^n}, \quad x,y\in \R, \ x\neq y.
        \end{equation}
        and
        \begin{equation}\label{3.6}
        \| \nabla_x \mathcal{C}(x,y,\cdot)\|_{L^\infty(0,\infty)}+\| \nabla_y \mathcal{C}(x,y,\cdot)\|_{L^\infty(0,\infty)}\leq \frac{C}{|x-y|^{n+1}}, \quad x,y\in \R, \ x\neq y.
        \end{equation}
        Indeed, by  \eqref{Omega} it follows that
        \begin{align*}
            \|\mathcal{C}(x,y,\cdot)\|_{L^\infty(0,\infty)}
                \leq & e^{-c(|x-y|^2+|y||x-y|)} \sup_{t>0}\int_0^\infty \frac{t+s}{(t+s+|x-y|)^{n+2}} ds
                \leq  C \sup_{t>0} \frac{e^{-c(|x-y|^2+|y||x-y|)}}{(t+|x-y|)^n} \\
                \leq  & C \frac{e^{-c(|x-y|^2+|y||x-y|)}}{|x-y|^n},
                \quad x,y\in \R, \ x \neq y.
        \end{align*}
        In order to show \eqref{3.6} we can proceed in a similar way.\\

        Suppose now that $f\in C_c^\infty(\R).$ We can write
        $$ G(f)(x,t)=\int_\R \mathcal{C}(x,y,t) f(y) dy, \quad x\not \in \supp f \text{ and } t>0.$$
        Let $x\not \in \supp f.$ Note that, for every $y\in \R,$ the function $g_{x,y}(t)=\mathcal{C}(x,y,t)f(y),$ $t\in(0,\infty)$
        is continuous, $\displaystyle \lim_{t\rightarrow \infty} g_{x,y}(t)=0$,  and there exists the limit $\displaystyle \lim_{t\rightarrow 0^+} g_{x,y}(t)$.\\

        We denote by $C_0([0,\infty))$ the space of continuous functions on $[0,\infty)$ that converge to zero in infinity. $C_0([0,\infty))$
        is endowed with the supremum norm. The dual space of $C_0([0,\infty))$ is the space of complex measures $\mathcal{M}([0,\infty))$ on $[0,\infty).$\\

        By \eqref{3.5} we have that $\int_\R \|\mathcal{C}(x,y,\cdot)\|_{L^\infty(0,\infty)}|f(y)|dy <\infty.$ We define
        $$ L_x(f)=\int_\R\mathcal{C}(x,y,\cdot)f(y)dy,$$
        where the last integral is understood in the $C_0([0,\infty))$-Bochner sense. Let $\mu\in \mathcal{M}([0,\infty)).$ We can write
        \begin{align*}
            \langle \mu ,L_x (f) \rangle_{\mathcal{M}([0,\infty)), C_0([0,\infty))}
                =&\int_{[0,\infty)}L_x(f)(s) d\mu(s)
                = \int_\R \int_{[0,\infty)} \mathcal{C}(x,y,s) d\mu(s) f(y) dy \\
                =& \int_{[0,\infty)} \int_\R \mathcal{C}(x,y,s) f(y)dy d\mu(s).
        \end{align*}
        Then,
        $$ L_x(f)(t)=\int_\R \mathcal{C}(x,y,t)f(y)dy, \quad t\in [0,\infty),$$
        and we conclude that
        $$ G(f)(x,t)=\left( \int_\R \mathcal{C}(x,y,\cdot) f(y)dy\right)(t), \quad t\in(0,\infty),$$
        where the integral is understood in the $C_0([0,\infty))$-Bochner sense.\\

        $R_{j,+}^\LL$ is bounded from $L^2(\R)$ into itself. Moreover, the maximal operator
        $$ P_*^{\LL + 2}(g)=\sup_{t>0} \left|P_t^{\LL+2}(g) \right|$$
        is also bounded from $L^2(\R)$ into  itself. Hence, $G$ is bounded operator from $L^2(\R)$ into $L^2(\R,$ $L^\infty(0,\infty)).$\\

        According to Banach valued Calder\'on-Zygmund theory we deduce that $G$ can be extended to $L^1(\R)$ as a bounded operator from $L^1(\R)$
        into $L^{1,\infty}(\R, L^\infty(0,\infty))$ and from $H^1(\R)$ into $L^1(\R, L^\infty(0,\infty)).$\\

        By proceeding now as in the final part of Subsection~\ref{subsec:2.2}, \eqref{3.5} allows us to conclude that $G$ can be extended to $H^1_\LL(\R)$
        as a bounded operator from $H^1_\LL(\R)$ into $L^1(\R, L^\infty(0,\infty))).$
        We denote by $\widetilde{G}$ to this extension.\\

        Suppose that $f\in H^1_\LL(\R).$ There exist a sequence $(a_j)_{j\in \mathbb{N}}$ of atoms for $H^1_\LL(\R)$ and a
        sequence $(\lambda_j)_{j\in \mathbb{N}}$ of complex numbers such that $\sum_{j=1}^\infty|\lambda_j| <\infty $
        and $f=\sum_{j=1}^\infty \lambda_j a_j.$ Since this series converge in $H^1_\LL(\R),$ we have that
        $$ \widetilde{G}(f)=\sum_{j=1}^\infty \lambda_j G(a_j), \quad \text{ in } L^1(\R, L^\infty(0,\infty)).$$
        Then, for every $t>0,$
        $$ \widetilde{G}(f)(\cdot,t)=\sum_{j=1}^\infty \lambda_j G(a_j)(\cdot,t)= \sum_{i=1}^\infty \lambda_i P_t^{\LL + 2}(R_{j,+}^\LL(a_i)), \quad \text{ in } L^1(\R).$$
        Moreover, for every $t>0,$
        $$ P_t^{\LL+2}R_{j,+}^\LL(f)=\sum_{i=1}^\infty \lambda_i P_t^{\LL +2}(R_{j,+}^\LL(a_i)), \quad \text{ in } L^1(\R).$$
        We conclude that
        $$ \widetilde{G}(f)(\cdot,t)=P_t^{\LL+2}(R_{j,+}^\LL(f)), \quad t>0,$$
        and the proof of this property is finished.
    \end{proof}

%\newpage
    %%%%%%%%%%%%%%%%%%%%%%%%%%%%%%%%%%%%%%%%%%%%%%%%%%%%%%%%%%%%%%%%%%%%%%%%%%%%%%%%%%%%%%%%%%%%%%%%%%%%%%%%%%%%%%%%%%%%

\end{document}